\documentclass[10pt,twoside]{article}

\usepackage{lmodern}
\usepackage[T1]{fontenc}
\usepackage[utf8]{inputenc}
\usepackage{textcomp}
\usepackage[english]{babel}

\newcommand{\fulltitle}{Asymptotic convergence of spectral inverse iterations for stochastic eigenvalue problems}
\newcommand{\shorttitle}{Asymptotic convergence of spectral inverse iterations for sEVPs}

\usepackage[top=31mm, bottom=31mm, left=33mm, right=33mm]{geometry}

\usepackage{fancyhdr}
\usepackage{graphicx}
\usepackage{hyperref}
\usepackage{amssymb}
\usepackage{amsthm}
\usepackage{amsmath}
\usepackage{verbatim}
\usepackage{subfig}
\usepackage{enumitem}

\def\L{\Lambda}

\def\k{\kappa}
\def\r{\rho}

\def\a{\alpha}
\def\b{\beta}
\def\g{\gamma}

\def\u{\mathbf{u}}

\def\vv{\mathbf{v}}
\def\w{\mathbf{w}}

\def\.{\cdot}

\def\x{\mathbf{x}}

\newcommand{\R}{\mathbb{R}}
\newcommand{\N}{\mathbb{N}}
\newcommand{\A}{\mathcal{A}}

\newcommand{\E}{\mathbb{E}}
\newcommand{\Var}{\mathrm{Var}}

\newcommand{\mbf}[1]{\mathbf{#1}}
\newcommand{\mbb}[1]{\mathbb{#1}}
\newcommand{\mc}[1]{\mathcal{#1}}
\newcommand{\W}{W_{\mc{A}}}
\newcommand{\supp}{\mathop{\mathrm{supp}}}

\newcommand{\hv}{\hat{\vv}}
\newcommand{\hu}{\hat{\u}}
\newcommand{\hw}{\hat{\mbf{w}}}

\newcommand{\sirv}{L^2_{\nu}(\Gamma)}
\newcommand{\fsmi}{(\N_0^{\infty})_c}

\DeclareMathOperator*{\esssup}{ess\,sup}
\DeclareMathOperator*{\essinf}{ess\,inf}
\DeclareMathOperator*{\dist}{dist}
\DeclareMathOperator*{\spa}{span}
\DeclareMathOperator*{\diag}{diag}

\newtheorem{thm}{Theorem}
\newtheorem{lmm}{Lemma}
\newtheorem{cor}{Corollary}
\newtheorem{prp}{Proposition}
\newtheorem{alg}{Algorithm}
\newtheorem{rmk}{Remark}

\title{}

\begin{document}

\title{\fulltitle}
\author{Harri Hakula \footnote{Aalto University, Department of Mathematics and Systems Analysis, P.O.\ Box 11100, FI-00076 Aalto, Finland; Email: harri.hakula@aalto.fi.} \and Mikael Laaksonen\footnote{Aalto University, Department of Mathematics and Systems Analysis, P.O.\ Box 11100, FI-00076 Aalto, Finland; Email: mikael.j.laaksonen@aalto.fi.}}
\date{June 8, 2017}

\maketitle

\begin{abstract}
We consider and analyze applying a spectral inverse iteration algorithm and its subspace iteration variant for computing eigenpairs of an elliptic operator with random coefficients. With these iterative algorithms the solution is sought from a finite dimensional space formed as the tensor product of the approximation space for the underlying stochastic function space, and the approximation space for the underlying spatial function space. Sparse polynomial approximation is employed to obtain the first one, while classical finite elements are employed to obtain the latter. An error analysis is presented for the asymptotic convergence of the spectral inverse iteration to the smallest eigenvalue and the associated eigenvector of the problem. A series of detailed numerical experiments supports the conclusions of this analysis. Numerical experiments are also presented for the spectral subspace iteration, and convergence of the algorithm is observed in an example case, where the eigenvalues cross within the parameter space. The outputs of both algorithms are verified by comparing to solutions obtained by a sparse stochastic collocation method.

\bigskip


\end{abstract}

\section{Introduction}

During the recent years numerical solution of stochastic partial differential equations (sPDE) has attracted a lot of attention and become a well-established field. However, the field of stochastic eigenvalue problems (sEVP) and their numerical solution is still in its infancy. It is natural that, after the source problem, more effort is put on addressing the eigenvalue problem.

A few different algorithms have recently been suggested for computing approximate eigenpairs of sEVPs. As with sPDEs, the solution methods are typically divided into intrusive and non-intrusive ones. A benchmark for non-intrusive methods is the sparse collocation algorithm suggested and thoroughly analyzed by Andreev and Schwab in \cite{andreevschwab12}. An attempt towards a Galerkin-based (intrusive) method was made by Verhoosel et al. in \cite{verhooselgutierrezhulshoff06}, though this method omits uniform normalization of the eigenmodes. Very recently Meidani and Ghanem proposed a spectral power iteration, in which the eigenmodes are normalized using a quadrature rule over the parameter space \cite{meidanighanem14}. The algorithm has been further developed and studied by Soused\'ik and Elman in \cite{sousedikelman16}. However, neither of the papers present a comprehensive error analysis for the method.

Inspired by the original method of Meidani and Ghanem we have suggested a purely Galerkin-based spectral inverse iteration, in which normalization of the eigenmodes is achieved via solution of a simple nonlinear system \cite{hakulakaarniojalaaksonen15}. This method, and its generalization to a spectral subspace iteration, is the focus of the current paper. Although the algorithms in \cite{meidanighanem14} and \cite{sousedikelman16} differ from ours in the way normalization is performed, the basic principles are still the same and hence our results on convergence should apply to these methods as well.

In this work we consider computing eigenpairs of an elliptic operator with random coefficients. We assume a physical domain $D \subset \R^d$ and, in order to capture the random dimension of the system, a parameter domain $\Gamma \subset \R^{\infty}$ with associated measure $\nu$. One may think of a parametrization that arises from Karhunen-Lo\`eve representations of the random coefficients in the system, for instance. Discretization in space is achieved by standard FEM and associated with a discretization parameter $h$, whereas discretization in the random dimension is achieved using collections of certain multivariate polynomials. These collections are represented by multi-index sets $\mc{A}_{\epsilon}$ of increasing cardinality $\# \mc{A}_{\epsilon}$ as $\epsilon \to 0$.

In the current paper we present a step-by-step analysis that leads to the main result: the asymptotic convergence of the spectral inverse iteration towards the exact eigenpair $(\mu, u)$. In this context the eigenpair of interest is the ground state, i.e., the smallest eigenvalue and the associated eigenfunction of the system. However, analogously to the classical inverse iteration, the computation of other eigenpairs may be possible by using a suitably chosen shift parameter $\lambda \in \R$. We show that under sufficient assumptions the iterates of the algorithm $(\mu_k, u_k)$ for $k = 1, 2, \ldots$ obey
\begin{equation}
\label{eq:totalconvergence1}
|| u - u_{h, \mc{A}, k} ||_{L^2_{\nu}(\Gamma) \otimes L^2(D)} \lesssim h^{1+l} + (\# \mc{A}_{\epsilon} )^{-r} + \lambda_{1/2}^k
\end{equation}
and
\begin{equation}
\label{eq:totalconvergence2}
|| \mu - \mu_{h, \mc{A}, k} ||_{L^2_{\nu}(\Gamma)} \lesssim h^{2l} +( \# \mc{A}_{\epsilon} )^{-r} + \lambda_{1/2}^k,
\end{equation}
where $l \in \N$ is the degree of polynomials used in the spatial discretization and $r > 0$ depends on the properties of the region to which the solution, as a function of the parameter vector, admits a complex-analytic extension. The quantity $\lambda_{1/2}$ reflects the gap between the two smallest eigenvalues of the system and should be less than one.

The first term in the formulas \eqref{eq:totalconvergence1} and \eqref{eq:totalconvergence2} is justified by standard theory for Galerkin approximation of eigenvalue problems, a simple consequence of which we have recapped in Theorem \ref{thm:herror}. The second term can be deduced from Theorem \ref{thm:fpconv}, which bounds the Galerkin approximation errors by residuals of certain polynomial approximations of the solution. Using best $P$-term polynomial approximations, we see that these residuals are ultimately expected to decay at an algebraic rate $r > 0$, see \cite{bieriandreevschwab09} and \cite{bierischwab09}. Finally, the third term follows from Theorem \ref{thm:iterconv}, which states that asymptotically the iterates of the spectral inverse iteration converge to a fixed point in geometric fashion. Here the analogy to classical inverse iteration is evident. Each of these three important steps that comprise the main result is separately verified through detailed numerical examples.

A variant of our algorithm for spectral subspace iteration is also presented. No analysis of this algorithm is given, but the numerical experiments support the conclusion that it convergences towards the exact subspace of interest, and that the rate of convergence is analogous to what we would expect from classical theory. This is despite the fact that the individual eigenmodes, as defined by the pointwise order of magnitude of the eigenvalues, are not continuous functions over the parameter space due to an eigenvalue crossing. To the authors' knowledge such a scenario has not yet been considered in the scientific literature.

The rest of the paper is organized as follows. Our model problem and its fundamental properties are assessed in Sections \ref{sec:problem} and \ref{sec:analyticity}. A detailed review of the discretization of the spatial and stochastic approximation spaces is given in Section \ref{sec:sfem}. Analysis of the spectral inverse iteration, supported by thorough numerical experiments, is given in Section \ref{sec:inverseiteration}. Finally, the algorithm of spectral subspace iteration and numerical experiments of its convergence are presented in Section \ref{sec:subspaceiteration}.

\section{Problem statement}
\label{sec:problem}

In this work we consider eigenvalue problems of elliptic operators with random coefficients. It is assumed that the random coefficients admit a parametrization with respect to countably many independent and bounded random variables. As a model problem we consider the eigenvalue problem of a diffusion operator with a random diffusion coefficient. It will be evident, however, that our methods and analysis in fact cover a much broader class of problems.

\subsection{Model problem}

Let $(\Omega, \mc{F}, \mc{P})$ be a probability space, $\Omega$ being the set of outcomes, $\mc{F}$ a $\sigma$-algebra of events, and $\mc{P}$ a probability measure defined on $\Omega$. We denote by $L^2_{\mc{P}}(\Omega)$ the space of square integrable random variables on $\Omega$ and define for $v \in L^2_{\mc{P}}(\Omega)$ the expected value
\[
\E[v] = \int_{\Omega} v(\omega) \ d\mc{P}(w)
\]
and variance $\Var[v] = \E[(v-\E[v])^2]$.

Let $D \subset \mbb{R}^d$ be a bounded convex domain with a sufficiently smooth boundary and assume a diffusion coefficient $a\!: D \times \Omega \to \mbb{R}$ that is a random field on $D$. The diffusion coefficient is assumed to be strictly uniformly positive and uniformly bounded, i.e., for some positive constants $a_{\min}$ and $a_{\max}$ it holds that
\begin{equation}
\label{eq:coefficientassumption}
\mc{P} \left( \omega \in \Omega : a_{\min} \le \essinf_{x \in D} a(x, \omega) \le \esssup_{x \in D} a(x, \omega) \le a_{\max} \right) = 1.
\end{equation}
We now formulate the model problem as: find functions $\mu\!: \Omega \to \mbb{R}$ and $u\!: D \times \Omega \to \mbb{R}$ such that the equations
\begin{equation}
\label{eq:modelproblem}
\left\{ \begin{array}{ll}
- \nabla \cdot (a(\cdot, \omega) \nabla u(\cdot, \omega)) = \mu(\omega) u(\cdot, \omega) & \textrm{ in } D \\
u(\cdot, \omega) = 0 & \textrm{ on } \partial D
\end{array} \right.
\end{equation}
hold $\mc{P}$-almost surely. In order to make the solutions physically meaningful we also impose a normalization condition $||u(\cdot, \omega)||_{L^2(D)} = 1$ that should hold $\mc{P}$-almost surely.

\subsection{Parametrization of the random input}

We make the assumption that the input random field admits a representation of the form
\begin{equation}
\label{eq:klexpansion}
a(x, \omega) = a_0(x) + \sum_{m = 1}^{\infty} a_m(x) y_m(\omega),
\end{equation}
where $\{y_m\}_{m = 1}^{\infty}$ are mutually independent and bounded random variables. For simplicity, we assume here that each $y_m$ is uniformly distributed. Thus, after possible rescaling, the dependence on $\omega$ is now parametrized by the vector $y = (y_1, y_2, \ldots) \in \Gamma := [-1,1]^{\infty}$. We denote by $\nu$ the underlying uniform product probability measure and by $L^2_{\nu}(\Gamma)$ the corresponding weighted $L^2$-space.

The usual convention is that the parametrization \eqref{eq:klexpansion} results from a Karhunen-Lo\`eve expansion, which gives $a(x, \omega)$ as a linear combination of the eigenfunctions of the associated covariance operator. The distinguishing feature of the Karhunen-Lo\`eve expansion compared to other linear expansions is that it minimizes the mean square truncation error \cite{ghanemspanos03}.

It is easy to see that $a_0 \in L^{\infty}(D)$ and
\begin{equation}
\label{eq:positiveness}
\essinf_{x \in D} a_0(x) > \sum_{m=1}^{\infty} || a_m ||_{L^{\infty}(D)}
\end{equation}
are sufficient conditions to ensure the assumption \eqref{eq:coefficientassumption}. In order to ensure analyticity of the eigenpair $(\mu, u)$ with respect to the parameter vector $y = (y_1, y_2, \ldots)$ we assume that
\begin{equation}
\label{eq:psummability1}
\sum_{m=1}^{\infty} || a_m ||^{p_0}_{L^{\infty}(D)} < \infty
\end{equation}
for some $p_0 \in (0,1)$ and that for a certain level of smoothness $s \in \N$ we have $a_0 \in W^{s,\infty}(D)$ and
\begin{equation}
\label{eq:psummability2}
\sum_{m=1}^{\infty} || a_m ||^{p_s}_{W^{s,\infty}(D)} < \infty
\end{equation}
for some $p_s \in (0,1)$. In particular, we consider the interesting case of algebraic
\[
|| a_m ||_{L^{\infty}(D)} \le C m^{-\varsigma}, \quad \varsigma > 1, \quad m = 1,2,\ldots
\]
decay of the coefficients in the series \eqref{eq:klexpansion}.

\subsection{Parametric eigenvalue problem and its variational formulation}

With the diffusion coefficient given by \eqref{eq:klexpansion}, the model problem \eqref{eq:modelproblem} becomes an eigenvalue problem of the operator
\[
(A(y)v)(x) := -\nabla \cdot (a(x,y) \nabla v(x)), \quad x \in D, \quad y \in \Gamma,
\]
where
\[
a(x, y) = a_0(x) + \sum_{m = 1}^{\infty} a_m(x) y_m.
\]
Thus, we obtain the parametric eigenvalue problem: find $\mu\!: \Gamma \to \mbb{R}$ and $u\!: \Gamma \to H_0^1(D)$ such that
\begin{equation}
\label{eq:parametricevp}
A(y)u(y) = \mu(y)u(y) \quad \forall y \in \Gamma.
\end{equation}
We denote by $\sigma(A(y))$ the set of eigenvalues of $A(y)$ for $y \in \Gamma$.

For any fixed $y \in \Gamma$ the problem \eqref{eq:parametricevp} reduces to a single deterministic eigenvalue problem. In variational form this is given by: find $\mu(y) \in \mbb{R}$ and $u(\cdot, y) \in H_0^1(D)$ such that
\begin{equation}
\label{eq:varform}
b(y; u(\cdot, y), v) = \mu(y) \langle u(\cdot, y), v \rangle_{L^2(D)} \quad \forall v \in H_0^1(D),
\end{equation}
where
\[
b(y; u, v) := \int_D a(\cdot ,y) \nabla u \cdot \nabla v \ dx.
\]
Under assumption \eqref{eq:positiveness} the bilinear form $b(y; u, v)$ is continuous and elliptic. Thus, as in \cite{hakulakaarniojalaaksonen15} and \cite{andreevschwab12}, we deduce that the problem \eqref{eq:varform} admits a countable number of real eigenvalues and corresponding eigenfunctions that form an orthogonal basis of $L^2(D)$.

\section{Analyticity of eigenmodes}
\label{sec:analyticity}

A key issue in the analysis of parametric eigenvalue problems is that eigenvalues may cross within the parameter space. Here we first disregard this possibility and recap the main results from \cite{andreevschwab12} for simple eigenvalues that are sufficiently well separated from the rest of the spectrum. After this we briefly comment on the case of possibly clustered eigenvalues and associated invariant subspaces.

\subsection{Eigenmodes of simple eigenvalues}

Let us first restrict our analysis to eigenvalues that are strictly nondegenerate. We call an eigenvalue $\mu$ of problem \eqref{eq:parametricevp} strictly nondegenerate if
\begin{itemize}
\item[(i)] $\mu(y)$ is simple as an eigenvalue of $A(y)$ for all $y \in \Gamma$ and
\item[(ii)] the minimum spectral gap $\inf_{y \in \Gamma} \dist(\mu(y), \sigma(A(y)) \backslash \{\mu(y) \})$ is positive.
\end{itemize}
In the case of strictly nondegenerate eigenvalues, the eigenpair $(\mu, u)$ is in fact analytic with respect to the parameter vector $y$.

\begin{prp}
\label{prp:analyticity}
Consider a strictly nondegenerate eigenvalue $\mu$ of the problem \eqref{eq:parametricevp} and the corresponding eigenfunction $u$ normalized so that $|| u(y) ||_{L^2(D)} = 1$ for all $y \in \Gamma$. For $s \in \N$ assume that $a_0 \in W^{s,\infty}(D)$ and the assumptions \eqref{eq:positiveness} -- \eqref{eq:psummability2} hold for some $p_0, p_s \in (0,1)$. Given $\tau = (\tau_1, \tau_2, \ldots)\in \mbb{R}_+^{\infty}$ define
\[
E(\tau) := \{ z \in \mbb{C}^{\infty} \ | \ \dist(z_m, [-1,1]) \le \tau_m \}.
\]
Then there exists $C_1 > 0$ independent of $m$ such that with $C_2 > 0$ arbitrary and $\tau$ given by
\[
\tau_m = \min \left\{ C_1 ||a_m||^{p_0-1}_{L^{\infty}(D)}, C_2 ||a_m||^{p_s-1}_{W^{s,\infty}(D)} \right\}, \quad m = 1,2, \ldots
\]
the eigenpair $(\mu,u)$ can be extended to a jointly complex-analytic function on $E(\tau)$ with values in $\mbb{C} \times (H^{s+1}(D) \cap H_0^1 (D))$.
\end{prp}

\begin{proof}
This is analogous to Corollary 2 of Theorem 4 in \cite{andreevschwab12}.
\end{proof}

It is well known that for elliptic operators on a connected domain $D$ the smallest eigenvalue is simple \cite{henrot06}. Thus, Proposition \ref{prp:analyticity} may at least be applied for the smallest eigenvalue of problem \eqref{eq:parametricevp}.

\subsection{Finite dimensional subspaces}

Let us now consider invariant subspaces for which the corresponding cluster of eigenvalues is sufficiently well separated from the rest of the spectrum. Assume a cluster $\mathcal{M}(y) = \{ \mu_q(y) \}_{q = 1}^Q$ of eigenvalues of \eqref{eq:parametricevp} so that
\begin{itemize}
\item[(i)] each  $\mu_q(y)$ is of finite multiplicity as an eigenvalue of $A(y)$ for all $y \in \Gamma$ and
\item[(ii)] the minimum spectral gap $\inf_{y \in \Gamma} \dist(\mathcal{M}(y), \sigma(A(y)) \backslash \mathcal{M}(y))$ is positive.
\end{itemize}
It is in general difficult to consider the analyticity of each of the eigenmodes separately. However, we might still expect the associated invariant subspace to be analytic as a function of $y$. More precisely, let $\{ u_q(y) \}_{q = 1}^{Q'}$ be a maximal collection of linearly independent eigenfunctions corresponding to the eigenvalues $\mathcal{M}(y)$ for all $y \in \Gamma$. It is not completely unreasonable to assume that $\spa \{ u_q(y) \}_{q = 1}^{Q'}$ is analytic, in a suitable sense, as a function of the parameter vector $y$. This assumption is the basis of our Algorithm of spectral subspace iteration given in Section \ref{sec:subspaceiteration}. For more information on the regularity of perturbed eigenvalues see \cite{kato76} and \cite{krieglmichorrainer11}.

\section{Stochastic finite elements}
\label{sec:sfem}

Proposition \ref{prp:analyticity}, under sufficient assumptions, guarantees the existence of an analytic eigenpair for problem \eqref{eq:parametricevp}. It now makes sense to look for the eigenvalue in the space $L^2_{\nu}(\Gamma)$ and the eigenfunction in the space $L^2_{\nu}(\Gamma) \otimes H^1_0(D)$. The space $H^1_0(D)$ may be discretized by means of the traditional finite element method. For the discretization of $L^2_{\nu}(\Gamma)$, we follow the usual convention in stochastic Galerkin methods and construct a basis of orthogonal polynomials of the input random variables. Orthogonal polynomials for various probability distributions exist and the use of these as the approximation basis has been observed to yield optimal rates of convergence \cite{soizeghanem04}, \cite{xiukarniadakis02}. Here we consider uniformly distributed random variables which lead to the choice of tensorized Legendre polynomials.

\subsection{Galerkin discretization in space}

Let $V_h \subset H_0^1(D)$ denote a finite dimensional approximation space associated with the discretization parameter $h > 0$. We assume approximation estimates
\begin{equation}
\label{eq:gestimate}
\inf_{v_h \in V_h} || v - v_h ||_{L^2(D)} \le Ch^{l+1}|| v ||_{H^{l+1}(D)}
\end{equation}
and
\begin{equation}
\label{eq:gestimate2}
\inf_{v_h \in V_h} || v - v_h ||_{H_0^1(D)} \le Ch^l|| v ||_{H^{l+1}(D)}
\end{equation}
that are standard for piecewise polynomials of degree $l$.

Fix $y \in \Gamma$ and let $(\mu_h, u_h)$ be the solution to the variational equation
\begin{equation}
\label{eq:dvarform}
b(y; u_h(\cdot, y), v_h) = \mu_h(y) \langle u_h(\cdot, y), v_h \rangle_{L^2(D)} \quad \forall v_h \in V_h,
\end{equation}
where $b(y; \cdot, \cdot)$ is as in \eqref{eq:varform}. Then we have the following bounds for the discretization error.
\begin{thm}
\label{thm:herror}
Assume \eqref{eq:gestimate} and \eqref{eq:gestimate2}. For $y \in \Gamma$ let $\mu(y)$ be a simple eigenvalue of \eqref{eq:varform} and $\mu_h(y)$ an eigenvalue of \eqref{eq:dvarform} such that $\lim_{h \to 0} \mu_h(y) = \mu(y)$. Let $u(\cdot, y) \in H^{1+l}(D)$ and $u_h(\cdot, y) \in V_h$ denote the associated eigenfunctions normalized in $L^2(D)$. Then there exists $C > 0$ such that
\begin{equation}
| \mu(y) - \mu_h(y) | \le C h^{2l},
\end{equation}
and
\begin{equation}
|| u(\cdot,y) - u_h(\cdot,y) ||_{L^2(D)} \le C h^{1+l}|| u(\cdot,y) ||_{H^{1+l}(D)}.
\end{equation}
as $h \to 0$.
\end{thm}

\begin{proof}
This follows from the theory of Galerkin approximation for variational eigenvalue problems. See Section 8 in \cite{babuskaosborn91} and Section 9 in \cite{boffi10}.
\end{proof}

Let $V_h = \spa \{ \varphi_i \}_{i \in J}$ where $J := \{1,2, \ldots, N \}$. Then \eqref{eq:dvarform} can be written as a parametric matrix eigenvalue problem: find $\mu_h \!: \Gamma \to \mbb{R}$ and $\u_h \!: \Gamma \to \mbb{R}^N$ such that
\begin{equation}
\label{eq:matrixevp}
\left( \mbf{K}^{(0)} + \sum_{m = 1}^{\infty} \mbf{K}^{(m)} y_m \right) \u_h (y) = \mu_h(y) \mbf{M} \u_h (y) \quad \forall y \in \Gamma,
\end{equation}
where $u_h(x,y) = \sum_{i \in J} \varphi_i(x) (\u_h)_i(y)$. The coefficient matrices are given by
\[
\mbf{K}^{(m)}_{ij} = \int_D a_m \nabla \varphi_i \cdot \nabla \varphi_j \ dx, \quad m = 0, 1, \ldots
\]
and
\[
\mbf{M}_{ij} = \int_D \varphi_i \varphi_j \ dx.
\]
For each fixed $y \in \Gamma$ the problem \eqref{eq:matrixevp} reduces to a positive-definite generalized matrix eigenvalue problem.

\subsection{Legendre chaos}

Recall that $y = (y_1, y_2, \ldots) \in \Gamma$ is a vector of mutually independent uniform random variables   and $\nu$ is the underlying constant product probablity measure. Now
\begin{equation}
\E[v] = \int_{\Gamma} v(y) \ d\nu(y)
\end{equation}
whenever the integral is finite. We define $(\mbb{N}_0^{\infty})_c$ to be the set of all multi-indices with finite support, i.e.,
\[
(\mbb{N}_0^{\infty})_c := \{ \a \in \mbb{N}_0^{\infty} \ | \ \# \supp(\a) < \infty \},
\]
where $\supp(\a) = \{ m \in \mbb{N} \ | \ \a_m \not= 0 \}$. Given a multi-index $\a \in (\mbb{N}_0^{\infty})_c$ we now define the multivariate Legendre polynomial
\[
\Lambda_{\a}(y) := \prod_{m \in \supp{\a}} L_{\a_m}(y_m), 
\]
where $L_p(x)$ denotes the univariate Legendre polynomial of degree $p$. We will assume the normalization $\E[\Lambda_{\a}^2] = 1$ for all $\a \in (\mbb{N}_0^{\infty})_c$.

The system $\{ \Lambda_{\a} (y) \ | \ \a \in (\mbb{N}_0^{\infty})_c \}$ forms an orthonormal basis of $L^2_{\nu}(\Gamma)$. Therefore, we may write any square integrable random variable $v$ in a series
\begin{align}
v(y) = \sum_{\a \in (\mbb{N}_0^{\infty})_c} v_{\a} \Lambda_{\a} (y)
\end{align}
with convergence in $L^2_{\nu}(\Gamma)$. The expansion coefficients are given by $v_{\a} = \E[v \Lambda_{\a}]$.

Due to the orthogonality of the Legendre polynomials we have $\E[\Lambda_{\a}] = \delta_{\a 0}$ and $\E[\Lambda_{\a} \Lambda_{\beta}] = \delta_{\a \beta}$ for all $\a, \beta \in (\mbb{N}_0^{\infty})_c$. Moreover, we denote
\begin{align*}
c_{\a \beta \gamma} & := \E[\Lambda_{\a} \Lambda_{\beta} \Lambda_{\gamma}], \quad \a, \beta, \gamma \in (\mbb{N}_0^{\infty})_c \\
c_{m \a \beta} & := \E[y_m \Lambda_{\a} \Lambda_{\beta}], \quad m \in \mbb{N}, \quad \a, \beta \in (\mbb{N}_0^{\infty})_c \\
c_{0 \a \beta} & := \delta_{\a \b}, \quad \a, \beta \in (\mbb{N}_0^{\infty})_c.
\end{align*}
We now obtain moment matrices $G^{(m)} \in \R^{P \times P}$ for $m \in \N_0$ and $G^{(\a)} \in \R^{P \times P}$ for $\a \in \mc{A}$ by setting $[G^{(m)}]_{\a \b} = c_{m \a \b}$ and $[G^{(\a)}]_{\b \g} = c_{\a \b \g}$.

\subsection{Sparse polynomial approximation in the parameter domain}

We fix a finite set $\mc{A} \subset (\mbb{N}_0^{\infty})_c$ and employ the approximation space $\W = \spa\{ \Lambda_{\a} \}_{\a \in \mc{A}} \subset L^2_{\nu} (\Gamma)$. We let $P_{\mc{A}}$ and $R_{\mc{A}}$ denote the underlying projection and residual operators so that $v \in L^2_{\nu}(\Gamma)$ is approximated by
\[
P_{\mc{A}} (v)(y) = \sum_{\a \in \mc{A}} v_{\a} \Lambda_{\a} (y)
\]
and the approximation error is given by $R_{\mc{A}}(v) = v - P_{\mc{A}} (v)$. Since
\begin{equation}
|| R_{\mc{A}}(v) ||^2_{\sirv} = \E \left[\left(\sum_{\a \in \A^c} v_{\a} \L_{\a}\right)^2 \right] = \sum_{\a \in \A^c} v^2_{\a},
\end{equation}
where $\A^c = \{ \a \in \fsmi \ | \ \a \notin \A \}$, we conclude that the choice of the multi-index set $\A$ ultimately determines the accuracy of our expansion.  

\begin{rmk}
The orthogonality of the basis $\{ \Lambda_{\a} \}_{\a \in \mc{A}}$ yields easy formulas for the mean and variance:
\begin{equation}
\E[P_{\mc{A}} (v)] = v_0, \quad \Var[P_{\mc{A}} (v)] = \left( \sum_{\a \in \mc{A}} v_{\a}^2 \right) - v_0^2.
\end{equation}
\end{rmk}

We proceed as in \cite{bieriandreevschwab09} and use best $P$-term approximations to prove convergence of the approximation error.

\begin{prp}
\label{prp:exponentialconv}
Let $H$ be a Hilbert space. Assume that $v \!: \Gamma \to H$ admits a complex-analytic extension in the region
\[
E(\tau) := \{ z \in \mbb{C}^{\infty} \ | \ \dist(z_m, [-1,1]) \le \tau_m \}
\]
with
\begin{equation}
\tau_m = C m^{\varrho}, \quad \varrho > 1, \quad m = 1,2, \ldots
\end{equation}
Then for each $\epsilon > 0$ and $0 < r < \varrho - \frac{1}{2}$ there exists $\mc{A}_{\epsilon} \subset (\mbb{N}_0^{\infty})_c$ such that
\begin{equation}
|| R_{\mc{A}_{\epsilon}}(v) ||_{L^2_{\nu}(\Gamma) \otimes H} \le \epsilon ||v||_{L^{\infty}(E(\tau);H)}
\end{equation}
and as $\epsilon \to 0$ we have
\begin{equation}
\# \mc{A}_{\epsilon} \le C(\varrho, r) \epsilon^{-1/r}.
\end{equation}
\end{prp}

\begin{proof}
This is essentially Proposition 3.1 in \cite{bieriandreevschwab09}, the proof of which can be found in \cite{bieriandreevschwab09II}. Here we recapitulate the main ideas.

Let $M \in \N$, $\mc{A} \subset \mbb{N}_0^M$ and $P = \# \mc{A}$. Set $v_M(z) := v(z_1, \ldots, z_M, 0, 0, \ldots)$ so that $P_{\mc{A}}(v) = P_{\mc{A}}(v_M)$ and
\begin{equation}
|| R_{\mc{A}}(v) ||_{L^2_{\nu}(\Gamma) \otimes H} = || v - v_M ||_{L^2_{\nu}(\Gamma) \otimes H} + || v_M - P_{\mc{A}}(v_M) ||_{L^2_{\nu}(\Gamma) \otimes H}.
\end{equation}
Define $\eta_m = \tau_m + \sqrt{1+\tau_m^2}$. It holds that
\begin{equation}
\sum_{m > M}^{\infty} (\eta_m-1)^{-1} \le C \sum_{m > M}^{\infty} m^{-\varrho} \le C \int_M^{\infty} x^{-\varrho} \ dx \le C (\varrho) M^{1- \varrho}
\end{equation}
and thus, by Lemmas 4.3. and 4.4 in \cite{babuskanobiletempone07}, we obtain
\begin{equation}
\label{eq:interpolationestimate}
|| v - v_M ||_{L^2_{\nu}(\Gamma) \otimes H} \le C ||v||_{L^{\infty}(E(\tau);H)} \sum_{m > M}^{\infty} (\eta_m-1)^{-1} \le C(\varrho) P^{-r} ||v||_{L^{\infty}(E(\tau);H)}
\end{equation}
for any $M \ge C(\varrho) P^{r/(\varrho-1)}$. On the other hand, the proof of Lemma A.3 in \cite{bieriandreevschwab09II} now applies to $v_M$, and we see that the Legendre coefficients (note the different scaling as compared to \cite{bieriandreevschwab09II})
\begin{equation}
(v_M)_{\a} = \int_{\Gamma} v_M(y) \Lambda_{\a}(y) \ d \nu(y)
\end{equation}
satisfy
\begin{equation}
|| (v_M)_{\a} ||_H \le C ||v_M||_{L^{\infty}(E(\tau);H)} \prod_{m=1}^M \eta_m^{-\a_m}.
\end{equation}
We set $\vartheta := (r + 1/2)^{-1}$ so that $\varrho^{-1} < \vartheta < 2$. Using the geometric series formula we compute
\begin{align}
\label{eq:coefficientestimate}
\sum_{\a \in \N^M_0} \prod_{m=1}^M \eta_m^{-\vartheta \a_m} & \le C \sum_{\a_1=0}^{\infty} \ldots \sum_{\a_M=0}^{\infty} \prod_{m=1}^M m^{-\varrho \vartheta \a_m} \nonumber \\
& = C \prod_{m=2}^M \left (\sum_{\a_m = 0}^{\infty} m^{-\varrho \vartheta \a_m} \right) \nonumber \\
& = C \prod_{m=2}^M \left( 1 +\frac{m^{-\varrho \vartheta}}{1-m^{-\varrho \vartheta}} \right) \nonumber \\
& \le C \exp \left( \sum_{m=2}^M \frac{m^{-\varrho \vartheta}}{1-m^{-\varrho \vartheta}} \right) \le C(\varrho, \vartheta),
\end{align}
where we have used the inequality $1 + x \le e^x$ for $x \ge 0$. Now assume that $v^* = \{ v^*_i \}_{i = 1}^{\infty}$ are the Legendre coefficients $\{ (v_M)_{\a} \}_{\a \in \N_0^M} \subset H$ of $v_M$ ordered by decreasing $H$-norm and that $\mc{A}$ represents a selection of the $P$ largest coefficients. Using \eqref{eq:coefficientestimate} we may now write
\begin{align}
\label{eq:legendreconvergence}
|| v_M - P_{\mc{A}} (v_M) ||^2_{L^2_{\nu}(\Gamma) \otimes H} & = \sum_{k > P} || v^*_k ||^2_H \nonumber \\
& \le \left( \sup_{k \ge 1} k^{\frac{1}{\vartheta}} || v^*_k ||_H \right)^2 \sum_{k > P} k^{-\frac{2}{\vartheta}} \nonumber \\
& \le \left( \sup_{k \ge 1} k || v^*_k ||_H^{\vartheta} \right)^{\frac{2}{\vartheta}} \int_P^{\infty} x^{-\frac{2}{\vartheta}} \ dx \nonumber \\
& \le \left(\sup_{k \ge 1} \sum_{i=1}^k || v^*_i ||_H^{\vartheta} \right)^{\frac{2}{\vartheta}} \left( \frac{2}{\vartheta} -1 \right)^{-1} P^{1 -\frac{2}{\vartheta}} \nonumber \\
& \le C(\varrho, r) P^{-2r} ||v_M||^2_{L^{\infty}(E(\tau);H)}.
\end{align}
The claim follows from combining \eqref{eq:interpolationestimate} and \eqref{eq:legendreconvergence}.
\end{proof}

\subsection{Stochastic Galerkin approximation of vectors and matrices}

We now generalize the concept of sparse polynomial approximation to vector and matrix valued functions. Assume that the dimensions of the approximation spaces $V_h$ and $\W$ are $N$ and $P$ respectively. We denote by $\W^N$ (or $\W^{N \times N}$) the space of functions $\vv \! : \Gamma \to \R^N$ (or $\mbf{A} \! : \Gamma \to \R^{N \times N}$) whose every component is in $\W$. Whenever $\vv \in \W^N$ and $\a \in \mc{A}$ we set $v_{\a i} = (\vv_i)_{\a}$ and use $\vv_{\alpha}$ to denote the vector of coefficients $\{ v_{\a i} \}_{i \in J} \in \R^N$. Moreover, we  associate any $v \in \W$ with the array of coefficients $\hat{v} := \{ v_{\a} \}_{\a \in \mc{A}} \in \R^P$ and similarly any $\vv \in \W^N$ with the array of coefficients $\hv := \{ v_{\a i} \}_{\a \in \mc{A}, i \in J} \in \R^{PN}$.

We denote by $\langle \cdot, \cdot \rangle_{\R^N_{\mbf{M}}}$ the inner product on $\R^N$ induced by the positive definite matrix $\mbf{M}$ and by $|| \cdot ||_{\R^N_{\mbf{M}}}$ the associated norm. Furthermore, we let $|| \cdot ||_{\R^P}$ denote the standard norm on $\R^P$ and $|| \cdot ||_{\R^P \otimes \R^N_{\mbf{M}}}$ denote the tensorized norm on $\R^{PN}$ given by
\[
|| \hv ||^2_{\R^P \otimes \R^N_{\mbf{M}}} := \sum_{\a \in \mc{A}} \sum_{i \in J} \sum_{j \in J} v_{\a i} \mbf{M}_{ij} v_{\a j}.
\]

\begin{rmk}
Observe that if $v \in \W \otimes V_h$ is written as $v(x,y) = \sum_{i \in J} \varphi_i(x) \vv_i(y)$, then
\begin{equation}
|| v ||_{L^2_{\nu}(\Gamma) \otimes L^2(D)}^2 = ||\vv||^2_{L^2_{\nu}(\Gamma) \otimes \R^N_{\mbf{M}}} = || \hv ||_{\R^P \otimes \R^N_{\mbf{M}}}^2.
\end{equation}
\end{rmk}

Let us consider the linear system defined by a parametric matrix $\mbf{A} \in \W^{N \times N}$. The Galerkin approximation of this system is: given $\mbf{f} \in \W^N$ find $\vv \in \W^N$ such that 
\begin{equation}
\label{eq:galerkinsystem}
P_{\mc{A}}(\mbf{A} \vv)(y) = \mbf{f}(y) \quad \forall y \in \Gamma.
\end{equation}
Using the notation introduced earlier we may write \eqref{eq:galerkinsystem} as the fully discrete system: given $\hat{\mbf{f}} \in \R^{PN}$ find $\hv \in \R^{P N}$ such that
\begin{equation}
\label{eq:coefficientmatrix}
\left( \sum_{\a \in \mc{A}} G^{(\a)} \otimes \mbf{A}_{\a} \right) \hv = \hat{\mbf{f}},
\end{equation}
where $\mbf{A}_{\a} = \E[\mbf{A} \Lambda_{\a}] \in \R^{N \times N}$. The existence of a solution, i.e. the invertibility of the coefficient matrix, is guaranteed by the following Lemma.

\begin{lmm}
\label{lmm:sinversion}
If $\mbf{A} \in \W^{N \times N}$ is a parametric matrix such that $\mbf{A}(y)$ is positive-definite for every $y \in \Gamma$, then for any $\mbf{f} \in \W^N$ there exists a unique $\vv \in \W^N$ such that 
\begin{equation}
\label{eq:slinearsys}
P_{\mc{A}}(\mbf{A} \vv)(y) = \mbf{f}(y) \quad \forall y \in \Gamma.
\end{equation}
Furthermore,
\begin{equation}
\label{eq:snormbound}
||\vv||_{L^2_{\nu}(\Gamma) \otimes \R^N_{\mbf{M}}} \le \sup_{y \in \Gamma} \lambda^{-1}(y) ||\mbf{f}||_{L^2_{\nu}(\Gamma) \otimes \R^N_{\mbf{M}}},
\end{equation}
where $\lambda(y)$ is the smallest eigenvalue of $\mbf{A}(y)$ for each $y \in \Gamma$.
\end{lmm}

\begin{proof}
Observe that the system \eqref{eq:slinearsys} is equivalent to the variational form
\begin{equation}
\label{eq:slinearvar}
\E[\vv^T \! \mbf{A} \mbf{w}] = \E[\mbf{f}^T \mbf{w}] \quad \forall \mbf{w} \in \W^N.
\end{equation}
The left hand side of \eqref{eq:slinearvar} is a symmetric and elliptic bilinear form so the existence of a unique solution is guaranteed by the Lax-Milgram Lemma. Moreover, the associated coefficient matrix in \eqref{eq:coefficientmatrix} is positive definite.

Now let $\tilde{\lambda} \in \R$ be such that $\tilde{\lambda} < \inf_{y \in \Gamma} \lambda(y)$. The matrix $\mbf{A}(y) - \tilde{\lambda} \mbf{I}_N$, where $\mbf{I}_N$ is an identity matrix, is positive definite for all $y \in \Gamma$. Thereby the eigenvalues of the associated coefficient matrix should be positive. Let $\chi$ be an eigenvalue of \eqref{eq:coefficientmatrix}, i.e., there exists $\w \in \W^N$ such that
\begin{equation}
P_{\mc{A}}(\mbf{A} \w)(y) = \chi \w(y) \quad \forall y \in \Gamma.
\end{equation}
Then
\begin{equation}
P_{\mc{A}}((\mbf{A} - \tilde{\lambda} \mbf{I}_N)\w)(y) = P_{\mc{A}}(\mbf{A}\w)(y) - \tilde{\lambda} \w(y) = (\chi - \tilde{\lambda})\w(y) \quad \forall y \in \Gamma
\end{equation}
and we deduce that $\chi > \tilde{\lambda}$. Equation \eqref{eq:snormbound} now follows from taking the limit $\tilde{\lambda} \to \inf_{y \in \Gamma} \lambda(y)$.
\end{proof}

\section{Spectral inverse iteration}
\label{sec:inverseiteration}

In this section we introduce the Algorithm of spectral inverse iteration, analyze its asymptotic convergence, and present numerical examples to support our analysis. The spectral inverse iteration, see \cite{hakulakaarniojalaaksonen15}, can be considered as an extension of the classical inverse iteration to the case of parametric matrix eigenvalue problems. In the spectral version each of the elementary operations is computed in Galerkin sense via projecting to the sparse polynomial basis $\W$. Optimal convergence of the Algorithm requires that the eigenmode of interest, i.e., the smallest eigenvalue of the parametric matrix, is strictly nondegenerate.

\subsection{Algorithm description}

Fix a finite set of multi-indices $\mc{A} \subset (\mbb{N}_0^{\infty})_c$ and let $P= \#\mc{A}$. The spectral inverse iteration for the system \eqref{eq:matrixevp} is now defined in Algorithm \ref{alg:sii}. One should note that, if the projections in the Algorithm were precise, the Algorithm would correspond to performing classical inverse iteration pointwise over the parameter space $\Gamma$. We expect the Algorithm to converge to an approximation of the eigenvector corresponding to the smallest eigenvalue of the system. 

\begin{alg}[Spectral inverse iteration]
\label{alg:sii}
Fix $tol > 0$ and let $\u^{(0)} \in \W^N$ be an initial guess for the eigenvector. For $k = 1,2,\ldots$ do
\begin{enumerate}
\item[(1)] Solve $\vv \in \W^N$ from the linear equation
\begin{equation}
\label{eq:stepone}
P_{\mc{A}} \left( \mbf{K} \vv \right) = \mbf{M} \u^{(k-1)}.
\end{equation}
\item[(2)] Solve $s \in \W$ from the nonlinear equation
\begin{equation}
\label{eq:steptwo}
P_{\mc{A}} (s^2) = P_{\mc{A}} \left( || \vv ||_{\R^N_{\mbf{M}}}^2 \right).
\end{equation}
\item[(3)] Solve $\u^{(k)} \in \W^N$ from the linear equation
\begin{equation}
\label{eq:stepthree}
P_{\mc{A}} \left( s \u^{(k)} \right) = \vv.
\end{equation}
\item[(4)] Stop if $|| \u^{(k)} - \u^{(k-1)} ||_{L^2_{\nu}(\Gamma) \otimes \R^N_{\mbf{M}}} < tol$ and return $\u^{(k)}$ as the approximate eigenvector.
\end{enumerate}
\end{alg}

Once the approximate eigenvector $\u^{(k)} \in \W^N$ has been computed, the corresponding eigenvalue $\mu^{(k)} \in \W$ may be evaluated from the Rayleigh quotient, as in \cite{hakulakaarniojalaaksonen15}, or alternatively from the linear system
\begin{equation}
\label{eq:eigenvalue}
P_{\mc{A}}(s \mu^{(k)}) = 1.
\end{equation}
Lemma \ref{lmm:sinversion} guarantees the invertibility of the linear system \eqref{eq:stepone} and, assuming that $s(y) > 0$ for all $y \in \Gamma$, the invertibility of the systems \eqref{eq:stepthree} and \eqref{eq:eigenvalue}. The nonlinear system \eqref{eq:steptwo} may be solved using for instance Newton's method.

\begin{rmk}
For the computation of non-extremal eigenmodes, one may proceed as in \cite{hakulakaarniojalaaksonen15} and replace $\mbf{K}(y)$ in \eqref{eq:stepone} with $(\mbf{K}(y) - \lambda \mbf{M})$, where $\lambda \in \R$ is a suitably chosen parameter. In this case we expect the algorithm to converge to an eigenpair for which the eigenvalue is close to $\lambda$. Note, however, that now the existence of a unique solution to \eqref{eq:stepone} is not necessarily guaranteed by Lemma \ref{lmm:sinversion}.
\end{rmk}

We try to write Algorithm \ref{alg:sii} in a computationally more convenient form. The projections in the algorithm can be computed explicitly using the notation introduced in Section \ref{sec:sfem}. It is easy to verify that equations \eqref{eq:stepone} -- \eqref{eq:stepthree} become
\begin{align}
\sum_{m = 0}^{\infty} \sum_{\b \in \mc{A}} \mbf{K}^{(m)} \vv_{\b} c_{m \a \b} & = \mbf{M} \u^{(k-1)}_{\a} \quad \forall \a \in \mc{A}, \\
\sum_{\b \in \mc{A}} \sum_{\g \in \mc{A}} s_{\b} s_{\g} c_{\a \b \g} & = \sum_{\b \in \mc{A}} \sum_{\g \in \mc{A}} \langle \vv_{\b}, \vv_{\g} \rangle_{\R^N_{\mbf{M}}} c_{\a \b \g} \quad \forall \a \in \mc{A}, \\
\sum_{\b \in \mc{A}} \sum_{\g \in \mc{A}} s_{\b} \u^{(k)}_{\g} c_{\a \b \g} & = \vv_{\a} \quad \forall \a \in \mc{A}
\end{align}
respectively. Given $\hat{s} = \{ s_{\a} \}_{\a \in \mc{A}} \in \R^P$ we define matrices
\begin{align*}
\Delta (\hat{s}) & := \sum_{\a \in \mc{A}} G^{(\a)} s_{\a}, \\
\widehat{\mbf{K}} & := \sum_{m=0}^{M(\mc{A})} G^{(m)} \otimes \mbf{K}^{(m)}, \\
\widehat{\mbf{M}} & := I_P \otimes \mbf{M}, \\
\mbf{S} & := \widehat{\mbf{M}}^{-1} \widehat{\mbf{K}}, \\
\mbf{T}(\hat{s}) & := \Delta(\hat{s}) \otimes \mbf{I}_N,
\end{align*}
where $M(\mc{A}) := \max \{ m \in \N \ | \ \exists \a \in \mc{A} \textrm{ s.t. } \a \not= 0 \}$ and $I_P \in \R^{P \times P}$ and $\mbf{I}_N \in \R^{N \times N}$ are identity matrices. We also define the nonlinear function $F \! : \R^{P} \times \R^{PN} \to \R^P$ via
\[
F_{\a}(\hat{s}, \hv) := \hat{s} \cdot G^{(\a)} \hat{s} - \hv \cdot (G^{(\a)} \otimes \mbf{M}) \hv, \quad \a \in \mc{A}
\]
and let $F^s \! : \R^P \times \R^P \to \R^P$ and $F^v \! : \R^{PN} \times \R^{PN} \to \R^P$ denote the associated bilinear forms given by $F^s_{\a} (\hat{s},\hat{t}) := \hat{s} \cdot G^{(\a)} \hat{t}$ and $F^v_{\a}(\hv, \hw) := \hv \cdot (G^{(\a)} \otimes \mbf{M}) \hw$. Now Algorithm \ref{alg:sii} may be rewritten in the following form.

\begin{alg}[Spectral inverse iteration in tensor form]
\label{alg:siit}
Fix $tol > 0$ and let $\hu^{(0)} = \{ u_{\a i}^{(0)} \}_{\a \in \mc{A}, i \in J} \in \R^{PN}$ be an initial guess for the eigenvector. For $k = 1,2,\ldots$ do
\begin{enumerate}
\item[(1)] Solve $\hv = \{ v_{\a i} \}_{\a \in \mc{A}, i \in J} \in \R^{PN}$ from the linear system
\begin{equation}
\label{eq:linearsys}
\widehat{\mbf{K}} \hv = \widehat{\mbf{M}} \hu^{(k-1)}.
\end{equation}
\item[(2)] Solve $\hat{s} = \{ s_{\a} \}_{\a \in \mc{A}} \in \R^P$ from the nonlinear system
\begin{equation}
\label{eq:nonlinearsys}
F(\hat{s},\hv) = 0
\end{equation}
with the initial guess $s_{\a} = || \hv ||_{\R^P \otimes \R^N_{\mbf{M}}} \delta_{\a 0}$ for $\a \in \mc{A}$.
\item[(3)] Solve $\hu^{(k)} = \{ u_{\a i}^{(k)} \}_{\a \in \mc{A}, i \in J} \in \R^{PN}$ from the linear system
\begin{equation}
\label{eq:blocklinearsys}
\mbf{T}(\hat{s}) \hu^{(k)} = \hv.
\end{equation}
\item[(4)] Stop if $|| \hu^{(k)} - \hu^{(k-1)} ||_{\R^P \otimes \R^N_{\mbf{M}}} < tol$ and return $\hu^{(k)}$ as the approximate eigenvector.
\end{enumerate}
\end{alg}

The approximate eigenvalue $\hat{\mu}^{(k)} \in \R^P$ may now be solved from the equation
\begin{equation}
\Delta(\hat{s}) \hat{\mu}^{(k)} = \hat{e}_1,
\end{equation}
where $\hat{e}_1 = \{ \delta_{\a 0} \}_{\a \in \mc{A}} \in \R^P$. 

\begin{rmk}
In \cite{hakulakaarniojalaaksonen15} Newton's method with the initial guess $s_{\a} = ||\vv_{\a}||_{\R^N_{\mbf{M}}}$ was suggested for the system of equations \eqref{eq:nonlinearsys}. Here the initial guess is somewhat different and corresponds to $s_0 = || \vv ||_{L^2_{\nu}(\Gamma) \otimes \R^N_{\mbf{M}}}$ (and $s_{\a} = 0$ for $\a \not= 0$).
\end{rmk}

In general it is not guaranteed that the Newton iteration for the system \eqref{eq:nonlinearsys} converges to a solution. The following proposition will give some insight to the conditions under which this happens to be the case.

\begin{prp}
\label{prp:newton}
Fix $\hv \in \R^{PN}$ and let $\hat{s}^{(0)} = \{ s^{(0)}_{\a} \}_{\a \in \mc{A}} \in \R^P$ be given by $s_{\a}^{(0)} = || \hv ||_{\R^P \otimes \R^N_{\mbf{M}}} \delta_{\a 0}$ for $\a \in \mc{A}$. Assume that there is a norm $|| \cdot ||_*$ on $\R^P$ and $r>0$ such that
\[
|| F^s (\hat{s},\hat{t}) ||_* \le C_F || \hat{s} ||_* || \hat{t} ||_*
\]
for all $\hat{s}, \hat{t}$ in $B(\hat{s}^{(0)}, r):= \{ \hat{s} \in \R^P \ | \ || \hat{s} - \hat{s}^{(0)} ||_* \le r \}$. If
\[
|| F(\hat{s}^{(0)},\hv) ||_* < C_F^{-1} || \hv ||_{\R^P \otimes \R^N_{\mbf{M}}}^2
\]
then the Newton method for $F(\cdot,\hv) = 0$ with the initial guess $\hat{s}^{(0)}$ converges to a unique solution in $B(\hat{s}^{(0)},r)$.
\end{prp}

\begin{proof}
This is a direct application of the Newton-Kantorovich theorem for the equation $F(\cdot,\hv) = 0$, see \cite{kantorovichakhilov64} (Theorem 6, 1.XVIII). Note that the first derivative (Jacobian) of $F(\cdot,\hv)$ at $\hat{s}^{(0)}$ is $2 || \hv ||_{\R^P \otimes \R^N_{\mbf{M}}} I_P$ and the second derivative is represented by the tensor of coefficients $2c_{\a \b \g}$.
\end{proof}

From Proposition \ref{prp:newton} we see that convergence of the Newton iteration is a consequence of the boundedness of the function $F^s$, which again is ultimately determined by the structure of the multi-index set $\mc{A}$.

\subsection{Analysis of convergence}

Due to a lack of general mathematical theory for multi-parametric eigenvalue problems we rely on a slightly unconventional approach in analyzing our algorithm. First of all, we restrict ourselves to asymptotic analysis since the underlying problem is nonlinear and thus hard to analyze globally. Second, we will analyze the solutions pointwise in the parameter space and deduce convergence theorems from classical eigenvalue perturbation bounds.

\subsubsection{Characterization of the dominant fixed point}

The classical inverse iteration converges to the dominant eigenpair of the inverse matrix. In a somewhat similar fashion the spectral inverse iteration tends to converge to a certain fixed point, which we shall refer to as the dominant fixed point. Here we will establish a connection between this dominant fixed point of the spectral inverse iteration and the dominant eigenpair of the inverse of the parametric matrix under consideration. This connection is obtained by considering the fixed point as a pointwise perturbation of the eigenvalue problem of the parametric matrix.

If $\u_{\mc{A}} \in \W^N$ is a fixed point of the Algorithm \ref{alg:sii}, then there exists a pair $(s, \vv) \in \W \times \W^N$ such that $\u_{\mc{A}} = \mbf{M}^{-1} P_{\mc{A}} (\mbf{K} \vv)$ and 
\begin{equation}
\label{eq:projection}
\left\{ \begin{array}{l}
P_{\mc{A}} \left(s P_{\mc{A}}(\mbf{K} \vv) \right) = \mbf{M} \vv \\
P_{\mc{A}} \left(s^2 - ||\vv||_{\R^N_{\mbf{M}}}^2 \right) = 0.
\end{array} \right.
\end{equation}
We call $\u_{\mc{A}}$ the dominant fixed point if, whenever $(\tilde{s}, \tilde{\vv}) \not= (s, \vv)$ also solves the system \eqref{eq:projection}, then $s(y) > \tilde{s}(y)$ for all $y \in \Gamma$. For any fixed $y \in \Gamma$ we may write \eqref{eq:projection} as
\begin{equation}
\label{eq:approximate}
\left\{ \begin{array}{l}
s(y) \mbf{K}(y) \vv(y) = \mbf{M}\vv(y) + s(y) R_{\mc{A}}(\mbf{K} \vv)(y) + R_{\mc{A}} (s P_{\mc{A}}(\mbf{K} \vv))(y) \\
s^2(y) = || \vv(y)||_{\R^N_{\mbf{M}}}^2 + R_{\mc{A}} \left(s^2 - || \vv ||_{\R^N_{\mbf{M}}}^2 \right) \! (y).
\end{array} \right.
\end{equation}
The following Lemma will be helpful in establishing a connection between the eigenpair of interest and the system \eqref{eq:approximate}.

\begin{lmm}
\label{lmm:perturbedsolution}
Denote by $|| \cdot ||$ the standard Euclidean norm on $\R^N$. Assume that $S \in \R^{N \times N}$ can be diagonalized as
\begin{equation}
S ( x \ X ) = (x \ X) \left( \begin{array}{cc} \lambda_1 & 0 \\ 0 & \Lambda \end{array} \right),
\end{equation}
where $\lambda_1 \in \R$, $\Lambda = \diag(\lambda_2, \ldots, \lambda_N)$ is real, and $(x \ X)$ is orthogonal. Assume also that $\lambda_1 > \lambda_2 \ge \ldots \ge \lambda_N$ and denote $\hat{\lambda} := \lambda_1 - \lambda_2$. Let $\r \in \R$ and $r \in \R^N$ be such that $|\r| \le 1/2$ and $|| r || \le \hat{\lambda}/8$. Then there exist $\k \ge 1/2$ and $\pi \in \R^{N-1}$ such that
\begin{itemize}
\item [(i)] The pair $(s, w)$ given by $s = \lambda_1 - \k^{-1} x^T r$ and $w = \k x + X \pi$ solves the system
\begin{equation}
\label{eq:psystem}
\left\{ \begin{array}{l} S w = s w + r \\ ||w||^2 = 1 + \rho. \end{array} \right.
\end{equation}
\item [(ii)] If $(\tilde{s}, \tilde{w}) \not= (s, w)$ also solves the system \eqref{eq:psystem}, then $s > \tilde{s}$ or $x^T \tilde{w} < 0$.
\item [(iii)] There exists $C > 0$ such that $| \k - 1 | \le C (|\r| + \hat{\lambda}^{-2} || r ||^2)$ and $|| \pi || \le C \hat{\lambda}^{-1} || r ||$.
\end{itemize}
\end{lmm}

\begin{proof}
\begin{itemize}
\item [(i)] Let $s(\k) = \lambda_1 - \k^{-1} x^T r$. For any $\k \ge 1/2$ we have $|\kappa^{-1}x^Tr| \le \hat{\lambda}/4$ so that
\begin{equation}
\min_{2 \le i \le N} | \lambda_i - s(\k) | = \min_{2 \le i \le N} | \lambda_1 - \lambda_i - \k^{-1} x^T r| \ge \hat{\lambda} - \frac14 \hat{\lambda} > \frac{1}{2} \hat{\lambda}
\end{equation}
and
\begin{equation}
||(\Lambda - s(\k)I)^{-1}|| \le 2 \hat{\lambda}^{-1}.
\end{equation}
The function
\begin{equation}
f(\k) = \k^2 + || (\Lambda - s(\k)I)^{-1} X^T r ||^2 - 1- \r
\end{equation}
is strictly increasing for $\k \ge 1/2$ since
\begin{align}
\k^2 f'(\k) & = 2\k^3 + 2 x^T r || (\Lambda - s(\k)I)^{-\frac{3}{2}} X^T r ||^2 \nonumber \\
& \ge 2 (\k^3 - (2 \hat{\lambda}^{-1})^3 || r ||^3) \nonumber \\
& > 2 (\k^3 - 2^{-3}) \ge 0.
\end{align}
One may also verify that $f(1/2) < 0$ and $f(2) > 0$. Thus, we may choose $\k > 1/2$ such that $f(\k) = 0$. For $w = \k x + X \pi$ we obtain 
\begin{equation}
Sw - sw = \k Sx + SX \pi - \k sx - s X \pi = \k (\lambda_1 - s) x + X (\Lambda - s I) \pi
\end{equation}
so the equation $Sw = sw + r$ is equivalent to
\begin{equation}
\left\{ \begin{array}{l} x^T (Sw - sw - r) = \k (\lambda_1 - s) - x^T r = 0 \\ X^T (Sw - sw - r) = (\Lambda - s I)\pi - X^T r = 0. \end{array} \right.
\end{equation}
Choosing $s = s(\k)$ and $\pi = (\Lambda - s I)^{-1} X^T r$ we see that both equations are satisfied. Moreover
\begin{equation}
||w||^2 = \k^2 + ||\pi||^2 = f(\k) + 1 + \r = 1 + \r.
\end{equation}
\item [(ii)] Suppose $(\tilde{s}, \tilde{w})$ also solves the system \eqref{eq:psystem} and write $\tilde{w} = \tilde{\k}x + X \tilde{\pi}$ for some $\tilde{\k} \in \R$ and $\tilde{\pi} \in \R^{N-1}$. In the nontrivial case we have $\tilde{\k} = x^T \tilde{w} > 0$. Assume first that $0 \le \tilde{\k} \le 1/2$. We have 
\begin{align}
\tilde{s} & = \frac{\tilde{w}^T S \tilde{w} - \tilde{w}^T r}{|| \tilde{w} ||^2} = \frac{\lambda_1 \tilde{\k}^2 + \tilde{\pi}^T \Lambda \tilde{\pi} - \tilde{w}^T r}{|| \tilde{w} ||^2} \le \frac{\lambda_1 \tilde{\k}^2 + \lambda_2 ||\tilde{\pi}||^2 + ||\tilde{w}|| ||r||}{|| \tilde{w} ||^2} \nonumber \\
& = \lambda_2 + \frac{\tilde{\k}^2}{1 + \r} \hat{\lambda} + \frac{||r||}{(1 + \r)^{\frac{1}{2}}}.
\end{align}
Since $s \ge \lambda_1 - \k^{-1} ||r||$, we deduce that
\begin{equation}
s - \tilde{s} \ge \hat{\lambda} - \k^{-1}||r|| - \frac{\tilde{\k}^2}{1 + \r} \hat{\lambda} - \frac{||r||}{(1 + \r)^{\frac{1}{2}}} > \left(1 - \frac14 - \frac12 - \frac{\sqrt{2}}8 \right) \hat{\lambda} > 0.
\end{equation}
Now let $\tilde{\k} \ge 1/2$. If $(\tilde{s},\tilde{w})$ is to solve \eqref{eq:psystem} then, as in part (i), we should have
\begin{equation}
\left\{ \begin{array}{l} \tilde{\k} (\lambda_1 - \tilde{s}) - x^T r = 0 \\ (\Lambda - \tilde{s} I)\tilde{\pi} - X^T r = 0. \end{array} \right.
\end{equation}
From the first equation we obtain $\tilde{s} = \lambda_1 - \tilde{\k}^{-1} x^T r$. Due to $|\tilde{\k}^{-1} x^T r| \le \hat{\lambda}/4$ the matrix $(\Lambda - \tilde{s}I)$ is invertible so the second equation gives $\tilde{\pi} = (\Lambda - \tilde{s} I)^{-1} X^T r$. Here $\tilde{\k} \ge 1/2$ must be chosen so that $f(\tilde{\k}) = 0$ and therefore $(\tilde{s},\tilde{w}) = (s,w)$.
\item[(iii)] From $f(\k) = 0$ and $\k \ge 1/2$ we deduce that
\begin{equation}
| \k - 1 | \le (\k+1)^{-1} (|\r| + || (\Lambda - s(\k)I)^{-1} X^T r ||^2) \le |\r| + 4 \hat{\lambda}^{-2}||r||^2
\end{equation}
and
\begin{equation}
||\pi|| = ||(\Lambda - s(\k)I)^{-1} X^T r || \le 2\hat{\lambda}^{-1} ||r||.
\end{equation}
\end{itemize}
\end{proof}

Applying Lemma \ref{lmm:perturbedsolution} to the system \eqref{eq:approximate} pointwise for $y \in \Gamma$ we obtain the following result.

\begin{prp}
\label{prp:pwestimate}
Let $\u_{\mc{A}} \in \W^N$ be the dominant fixed point of Algorithm \ref{alg:sii} and denote by $(s, \vv)$ the associated pair in $\W \times \W^N$ that solves \eqref{eq:projection}. Let $\mu_{\mc{A}} \in \W$ be such that $P_{\mc{A}}(s \mu_{\mc{A}}) = 1$. For $y \in \Gamma$ denote by $\hat{\lambda}(y)$ the gap between the two largest eigenvalues of $\mbf{K}^{-1}(y) \mbf{M}$. Assume that $\inf_{y \in \Gamma} s(y) > 0$ and $\inf_{y \in \Gamma} \hat{\lambda}(y) > 0$. For $y \in \Gamma$ define
\[
\mbf{r}(y) := \mbf{K}^{-1}(y) R_{\mc{A}} (\mbf{K} \vv)(y) + s^{-1}(y) \mbf{K}^{-1}(y) R_{\mc{A}} (s P_{\mc{A}} (\mbf{K} \vv))(y)
\]
and
\[
\r(y) := s^{-2}(y) R_{\mc{A}} \left(s^2 - || \vv ||_{\R^N_{\mbf{M}}}^2 \right) \! (y).
\]
If
\[
r_* := \sup_{y \in \Gamma} \hat{\lambda}^{-1}(y) || \mbf{r}(y) ||_{\R^N_{\mbf{M}}} < \frac{1}{8}
\]
and
\[
\r_* := \sup_{y \in \Gamma} | \r(y) | < \frac{1}{2}
\]
then there exists $C > 0$ such that
\begin{equation}
| \mu_{\mc{A}}(y) - \mu_h(y) | \le C \left(\max\{\mu_h^2(y), s^{-2}(y)\} || \mbf{r}(y) ||_{\R^N_{\mbf{M}}} + s^{-1}(y) | R_{\mc{A}}(s \mu_{\mc{A}}) (y)| \right)
\end{equation}
and
\begin{align}
|| & \u_{\mc{A}}(y) - \u_h(y) ||_{\R^N_{\mbf{M}}} \nonumber \\
& \le C \left(| \r(y) | + \hat{\lambda}^{-1}(y) || \mbf{r}(y) ||_{\R^N_{\mbf{M}}} + s^{-1}(y) || \mbf{M}^{-1} R_{\mc{A}} (s P_{\mc{A}}(\mbf{K} \vv))(y) ||_{\R^N_{\mbf{M}}} \right) \! ,
\end{align}
where $\mu_h \!: \Gamma \to \R$ is the smallest eigenvalue of $\mbf{M}^{-1}\mbf{K}(y)$ and $\u_h \!: \Gamma \to \R^N$ is the corresponding eigenvector normalized in $|| \cdot ||_{\R^N_{\mbf{M}}}$ (and with appropriate sign).
\end{prp}

\begin{proof}
It is easy to see that the system \eqref{eq:approximate} is equivalent to
\begin{equation}
\left\{ \begin{array}{l}
\mbf{M}^{\frac12} \mbf{K}^{-1}(y) \mbf{M}^{\frac12} \w(y) = s(y) \w(y) - \mbf{M}^{\frac12} \mbf{r}(y) \\
|| \w(y)||_{\R^N}^2 = 1 - \r(y),
\end{array} \right.
\end{equation}
where $\w(y) = s^{-1}(y) \mbf{M}^{\frac12} \vv(y)$. By Lemma \ref{lmm:perturbedsolution} the solution with the pointwise largest $s(y)$ can be written as
\begin{equation}
\left\{
\begin{array}{l}
s(y) = \lambda_1(y) + \k^{-1}(y) \x^T (y) \mbf{M}^{\frac12} \mbf{r}(y) \\
\w(y) = \k(y) \x(y) + \mbf{X}(y) \boldsymbol{\pi}(y),
\end{array}
\right.
\end{equation}
where $\k \!: \Gamma \to [1/2,\infty)$ and $\boldsymbol{\pi} \!: \Gamma \to \R^{N-1}$ are such that
\begin{equation*}
| \k(y) - 1 |^2 + || \boldsymbol{\pi}(y) ||_{\R^N}^2 \le C \left(|\r(y)| + \hat{\lambda}^{-2}(y) || \mbf{r}(y) ||_{\R^N_{\mbf{M}}}^2\right)^2 + C \hat{\lambda}^{-2}(y) || \mbf{r}(y) ||_{\R^N_{\mbf{M}}}^2,
\end{equation*}
$\lambda_1(y) = \mu_h^{-1}(y)$ is the pointwise largest eigenvalue of $\mbf{S}(y)$ and $\x(y) = \mbf{M}^{\frac12} \u_h(y)$ is the corresponding eigenvector. The matrix $(\x(y) \ \mbf{X}(y))$ is orthonormal for every $y \in \Gamma$. A Taylor expansion of $s^{-1}(y)$ yields
\begin{align}
| s^{-1}(y) - \mu_h(y) | & = C (\mu_h^{-1}(y) + \xi(y))^{-2} || \mbf{r}(y) ||_{\R^N_{\mbf{M}}} \nonumber \\
& \le C \max\{\mu_h^2(y), s^{-2}(y)\} || \mbf{r}(y) ||_{\R^N_{\mbf{M}}},
\end{align}
where $\xi(y)$ is such that $0 \le \xi(y) \le \k^{-1}(y) \x^T(y) \mbf{M}^{\frac12} \mbf{r}(y)$. Combining this with the equation
\begin{equation}
\mu_{\mc{A}}(y) = s^{-1}(y) + s^{-1}(y) R_{\mc{A}}(s \mu_{\mc{A}})(y)
\end{equation}
obtained from the condition $P_{\mc{A}}(s \mu_{\mc{A}}) = 1$, we have altogether that
\begin{equation}
| \mu_{\mc{A}}(y) - \mu_h(y) | \le C \max\{\mu_h^2(y), s^{-2}(y)\} ||\mbf{r}(y)||_{\R^N_{\mbf{M}}} + s^{-1}(y) | R_{\mc{A}}(s \mu_{\mc{A}})(y)|.
\end{equation}
Furthermore,
\begin{equation}
\mbf{M}^{\frac12} \u_{\mc{A}}(y) = \mbf{M}^{-\frac12}P_{\mc{A}}(\mbf{K} \vv)(y) = s^{-1}(y) \mbf{M}^{-\frac12} R_{\mc{A}}(sP_{\mc{A}}(\mbf{K} \vv))(y) + \w(y)
\end{equation}
from which it follows that
\begin{align}
|| \u_{\mc{A}}(y) & - \u_h(y) ||_{\R^N_{\mbf{M}}}^2 = || \mbf{M}^{\frac12} \u_{\mc{A}}(y) - \w(y) ||_{\R^N}^2 + || \w(y) -  \mbf{M}^{\frac12} \u_h(y) ||_{\R^N}^2 \nonumber \\
&= || s^{-1}(y) \mbf{M}^{-\frac12} R_{\mc{A}}(sP_{\mc{A}}(\mbf{K} \vv))(y) ||_{\R^N}^2 + || (\k(y) - 1) \x(y) + \mbf{X}(y) \boldsymbol{\pi}(y) ||_{\R^N}^2
\nonumber \\
&= s^{-1}(y) || \mbf{M}^{-1} R_{\mc{A}}(sP_{\mc{A}}(\mbf{K} \vv))(y) ||_{\R^N_{\mbf{M}}}^2 + |\k(y) - 1|^2 + || \boldsymbol{\pi}(y) ||_{\R^N}^2 \nonumber \\
& \le C \left(|\r(y)| + \hat{\lambda}^{-1}(y) || \mbf{r}(y) ||_{\R^N_{\mbf{M}}} + s^{-1}(y)|| \mbf{M}^{-1} R_{\mc{A}}(sP_{\mc{A}}(\mbf{K} \vv))(y) ||_{\R^N_{\mbf{M}}} \right)^2 \!.
\end{align}
\end{proof}

\begin{rmk}
Note that we have not proven the existence of a dominant fixed point of the Algorithm \ref{alg:sii}. The residuals $\mbf{r}$ and $\r$ in Proposition \ref{prp:pwestimate} depend on the pair $(s, \vv) \in \W \times \W^N$ and hence Lemma \ref{lmm:perturbedsolution} by itself is not sufficient to guarantee the existence of a dominant fixed point.
\end{rmk}

\subsubsection{Convergence of the dominant fixed point to a parametric eigenpair}

The next step in our analysis is to bound the error between the dominant fixed point of Algorithm \ref{alg:sii} and the dominant eigenpair of the inverse of the parametric matrix. To this end we will use the pointwise estimate obtained previously.

From Proposition \ref{prp:pwestimate} we may easily deduce the following result.

\begin{thm}
\label{thm:fpconv}
Let $\u_{\mc{A}} \in \W^N$ be the dominant fixed point of Algorithm \ref{alg:sii} and denote by $(s, \vv)$ the associated pair in $\W \times \W^N$ that solves \eqref{eq:projection}. Let $\mu_{\mc{A}} \in \W$ be such that $P_{\mc{A}}(s \mu_{\mc{A}}) = 1$. For $y \in \Gamma$ denote by $\hat{\lambda}(y)$ the gap between the two largest eigenvalues of $\mbf{K}^{-1}(y)\mbf{M}$. Assume that $s_* := \inf_{y \in \Gamma} s(y) > 0$, $\hat{\lambda}_* := \inf_{y \in \Gamma} \hat{\lambda}(y) > 0$ and that the quantity
\[
|| R_{\mc{A}}(\mbf{K} \vv)||_{L^{\infty}(\Gamma) \otimes \R^N_{\mbf{M}}} + || R_{\mc{A}} \left(s P_{\mc{A}}(\mbf{K} \vv)\right) ||_{L^{\infty}(\Gamma) \otimes \R^N_{\mbf{M}}} + \left| \left|R_{\mc{A}}\left(s^2 - || \vv ||_{\R^N_{\mbf{M}}}^2 \right) \right| \right|_{L^{\infty}(\Gamma)}
\]
is small enough. Then there exists $C > 0$ such that
\begin{align}
\label{eq:evalfperror}
|| \mu_{\mc{A}} - \mu_h ||_{L^2_{\nu}(\Gamma)} \le C & \left( || R_{\mc{A}}(\mbf{K} \vv) ||_{L^2_{\nu}(\Gamma) \otimes \R^N_{\mbf{M}}} \right. \nonumber \\
& \left. \quad + || R_{\mc{A}} (s P_{\mc{A}}(\mbf{K} \vv)) ||_{L^2_{\nu}(\Gamma) \otimes \R^N_{\mbf{M}}} + \left| \left| R_{\mc{A}}\left(s\mu_{\mc{A}}\right) \right|\right|_{L^2_{\nu}(\Gamma)} \right)
\end{align}
and
\begin{align}
\label{eq:evecfperror}
|| \u_{\mc{A}} - \u_h ||_{L^2_{\nu}(\Gamma) \otimes \R^N_{\mbf{M}}} \le C & \left( || R_{\mc{A}}(\mbf{K} \vv) ||_{L^2_{\nu}(\Gamma) \otimes \R^N_{\mbf{M}}} + || R_{\mc{A}} (s P_{\mc{A}}(\mbf{K} \vv)) ||_{L^2_{\nu}(\Gamma) \otimes \R^N_{\mbf{M}}} \right. \nonumber \\
& \left. \quad + \left| \left| R_{\mc{A}}\left(s^2 - || \vv ||_{\R^N_{\mbf{M}}}^2\right) \right|\right|_{L^2_{\nu}(\Gamma)} \right) \!,
\end{align}
where $\mu_h \!: \Gamma \to \R$ is the smallest eigenvalue of $\mbf{M}^{-1}\mbf{K}(y)$ and $\u_h \!: \Gamma \to \R^N$ is the corresponding eigenvector normalized in $|| \cdot ||_{\R^N_{\mbf{M}}}$ (and with appropriate sign). Here $C$ depends only on $s_*$, $\hat{\lambda}_*$, $\mu_h^* := \sup_{y \in \Gamma} \mu_h(y)$, $K_* = \sup_{y \in \Gamma} || \mbf{K}^{-1}(y)||_{\R^N_{\mbf{M}}}$, and $M_* = || \mbf{M}^{-1}||_{\R^N_{\mbf{M}}}$.
\end{thm}

\begin{proof}
With $\mbf{r}$ defined as in Proposition \ref{prp:pwestimate} we have
\begin{equation}
|| \mbf{r} ||_{L^2_{\nu}(\Gamma) \otimes \R^N_{\mbf{M}}} \le K_* \left( || R_{\mc{A}}(\mbf{K} \vv) ||_{L^2_{\nu}(\Gamma) \otimes \R^N_{\mbf{M}}} + s_*^{-1} || R_{\mc{A}} (s P_{\mc{A}}(\mbf{K} \vv)) ||_{L^2_{\nu}(\Gamma) \otimes \R^N_{\mbf{M}}} \right),
\end{equation}
and
\begin{equation}
|| \mbf{M}^{-1} R_{\mc{A}} (s P_{\mc{A}}(\mbf{K} \vv))(y) ||_{L^2_{\nu}(\Gamma) \otimes \R^N_{\mbf{M}}} \le M_* || R_{\mc{A}} (s P_{\mc{A}}(\mbf{K} \vv))(y) ||_{L^2_{\nu}(\Gamma) \otimes \R^N_{\mbf{M}}}.
\end{equation}
The bounds \eqref{eq:evalfperror} and \eqref{eq:evecfperror} now follow from Proposition \ref{prp:pwestimate}.
\end{proof}

By Proposition \ref{prp:analyticity} the exact eigenvalue and eigenvector of problem \eqref{eq:parametricevp} are analytic functions of the parameter vector $y \in \Gamma$. This suggests that the residuals on the right hand side of equations \eqref{eq:evalfperror} and \eqref{eq:evecfperror} can be asymptotically estimated from Proposition \ref{prp:exponentialconv}.

\subsubsection{Convergence of the spectral inverse iteration to the dominant fixed point}

The classical inverse iteration converges to the dominant eigenpair of the inverse matrix at a speed characterized by the gap between the two largest eigenvalues. Here we will establish a similar asymptotic result for the convergence of the spectral inverse iteration towards the dominant fixed point.

Fixed points of the spectral inverse iteration may be characterized using the tensor notation of Algorithm \ref{alg:siit}. Let $\hu_{\mc{A}} \in \R^{PN}$ be a fixed point of the algorithm, i.e., $\hu_{\mc{A}} = \mbf{S}\hv$ and $(\hat{s}, \hv) \in \R^P \times \R^{PN}$ are such that
\begin{equation}
\label{eq:fixedp}
\left\{ \begin{array}{l}
\hv = \mbf{T}(\hat{s}) \mbf{S} \hv \\
F(\hat{s}, \hv) = 0.
\end{array} \right.
\end{equation}
Define a linear operator $\mbf{R} (\hat{s}, \hv) \! : \R^{PN} \to \R^{PN}$ by
\[
\mbf{R} (\hat{s}, \hv) \hw := \hw - \mbf{T}\left(\Delta^{-1}(\hat{s}) F^v(\hv, \hw) \right) \mbf{T}^{-1}(\hat{s}) \hv.
\]
The convergence of the spectral inverse iteration to the fixed point $\hu_{\mc{A}}$ can now be related to the ratio of the norms of $\Delta^{-1}(\hat{s})$ and $\mbf{R} (\hat{s}, \hv)\mbf{S}^{-1}$.

\begin{thm}
\label{thm:iterconv}
Let $\hu_{\mc{A}} \in \R^{PN}$ be a fixed point of the Algorithm \ref{alg:siit} and $(\hat{s}, \hv) \in \R^P \times \R^{PN}$ a corresponding solution to \eqref{eq:fixedp}. Assume that $\Delta(\hat{s})$ is invertible. Let $\hat{\mu}_{\mc{A}} = \Delta^{-1}(\hat{s}) \hat{e}_1$, where $\hat{e}_1 = \{ \delta_{\a 0} \}_{\a \in \mc{A}} \in \R^P$. Set $\phi_{\min} := || \Delta^{-1}(\hat{s}) ||^{-1}_{\R^P}$ and $\psi_{\max} := ||\mbf{R} (\hat{s}, \hv) \mbf{S}^{-1}||_{\R^P \otimes \R^N_{\mbf{M}}}$. Then the iterates of Algorithm \ref{alg:siit} satisfy
\begin{equation}
\label{eq:eveciterconv}
|| \hu^{(k)} - \hu_{\mc{A}} ||_{\R^P \otimes \R^N_{\mbf{M}}} \le \frac{\psi_{\max}}{ \phi_{\min} } || \hu^{(k-1)} - \hu_{\mc{A}} ||_{\R^P \otimes \R^N_{\mbf{M}}}, \quad k \in \N
\end{equation}
for $\hu^{(k)}$ sufficiently close to $\hu_{\mc{A}}$. Furthermore, there exists $C > 0$ such that
\begin{equation}
\label{eq:evaliterconv}
|| \hat{\mu}^{(k)} - \hat{\mu}_{\mc{A}} ||_{\R^P} \le C || \hu^{(k)} - \hu_{\mc{A}} ||_{\R^P \otimes \R^N_{\mbf{M}}}, \quad k \in \N.
\end{equation}
\end{thm}

\begin{proof}
The partial derivative (Jacobian) of the function $F(\hat{s} + \hat{t}, \hv + \hw)$ with respect to $\hat{t}$ at $\hat{t} = 0$ is given by $2\Delta(\hat{s})$. The implicit function theorem now guarantees that there is a unique differentiable function $\hat{t}(\hw)$ defined in a neighbourhood of $\hw = 0$ such that $F(\hat{s} + \hat{t}(\hw),\hv + \hw) = 0$. Computing the first order approximation of this function we see that for $\hw$ small enough
\begin{equation}
\label{eq:linearization}
\hat{t}(\hw) = \Delta^{-1}(\hat{s}) F^v(\hv, \hw) + \textrm{h.o.t. in } \hw,
\end{equation}
where h.o.t. stands for higher order terms. From \eqref{eq:linearization} we obtain
\begin{align}
\mbf{T}^{-1}\left(\hat{s} + \hat{t}(\w)\right) & = \left(\mbf{T}(\hat{s}) + \mbf{T}\left(\Delta^{-1}(\hat{s}) F^v(\hv, \hw)\right) + \textrm{h.o.t. in } \hw \right)^{-1} \nonumber \\
& = \mbf{T}^{-1}(\hat{s}) - \mbf{T}^{-1}(\hat{s}) \mbf{T} \left(\Delta^{-1}(\hat{s}) F^v(\hv, \hw) \right) \mbf{T}^{-1}(\hat{s}) + \textrm{h.o.t. in } \hw.
\end{align}
Set $\hv^{(k)} = \mbf{S}^{-1} \hu^{(k)}$ and $\hw^{(k)} = \hv^{(k)} - \hv$. Now
\begin{align}
\mbf{S} \hw^{(k)} & = \mbf{T}^{-1}\left(\hat{s} + \hat{t}(\w^{(k-1)})\right)\left(\hv + \hw^{(k-1)}\right) - \mbf{S} \hv \nonumber \\
& = \mbf{T}^{-1}(\hat{s}) \left(\hv + \hw^{(k-1)} \right) - \mbf{T}^{-1}(\hat{s}) \mbf{T}\left(\Delta^{-1}(\hat{s})F^v(\hv, \hw^{(k-1)})\right) \mbf{T}^{-1}(\hat{s}) \hv - \mbf{S} \hv \nonumber \\
& \quad + \textrm{h.o.t. in } \hw^{(k-1)} \nonumber \\
& = \mbf{T}^{-1}(\hat{s})\left(\hw^{(k-1)} - \mbf{T}\left(\Delta^{-1}(\hat{s})F^v(\hv, \hw^{(k-1)})\right)  \mbf{T}^{-1}(\hat{s}) \hv \right) + \textrm{h.o.t. in } \hw^{(k-1)} \nonumber \\
& = \mbf{T}^{-1}(\hat{s})\mbf{R}(\hat{s}, \hv) \hw^{(k-1)} + \textrm{h.o.t. in } \hw^{(k-1)}.
\end{align}
Since $\mbf{S} \hw^{(k)} = \hu^{(k)} - \hu_{\mc{A}}$ we have that
\begin{equation}
\label{eq:linearupdate}
\hu^{(k)} - \hu_{\mc{A}} = \mbf{T}^{-1}(\hat{s})\mbf{R}(\hat{s}, \hv) \mbf{S}^{-1} \left(\hu^{(k-1)} - \hu_{\mc{A}} \right) + \textrm{h.o.t. in } \left(\hu^{(k-1)} - \hu_{\mc{A}} \right).
\end{equation}
Equations \eqref{eq:eveciterconv} and \eqref{eq:evaliterconv} now follow from \eqref{eq:linearupdate} and the fact that $\hat{\mu}^{(k)}$ is asymptotically given as a linear function of $\hu^{(k)}$.
\end{proof}

Adapting Theorem \ref{thm:iterconv} to the context of Algorithm \ref{alg:sii} we obtain the following Corollary.

\begin{cor}
\label{cor:iterconv}
Let $\u_{\mc{A}} \in \W^N$ be a fixed point of the Algorithm \ref{alg:sii} and $(s, \vv) \in \W \times \W^N$ a corresponding solution to \eqref{eq:projection}. Let $\mu_{\mc{A}} \in \W$ be such that $P_{\mc{A}}(s \mu_{\mc{A}}) = 1$. Assume that $s_* := \inf_{y \in \Gamma} s(y) > 0$ and let $\psi_{\max}$ be as in Theorem \ref{thm:iterconv}. Then the iterates of Algorithm \ref{alg:sii} satisfy
\[
|| \u^{(k)} - \u_{\mc{A}} ||_{L^2_{\nu}(\Gamma) \otimes \R^N_{\mbf{M}}} \le \frac{\psi_{\max}}{s_*}  || \u^{(k-1)} - \u_{\mc{A}} ||_{L^2_{\nu}(\Gamma) \otimes \R^N_{\mbf{M}}}, \quad k \in \N
\]
for $\u^{(k)}$ sufficiently close to $\u_{\mc{A}}$. Furthermore, there exists $C > 0$ such that
\begin{equation}
|| \mu^{(k)} - \mu_{\mc{A}} ||_{L^2_{\nu}(\Gamma)} \le C || \u^{(k)} - \u_{\mc{A}} ||_{L^2_{\nu}(\Gamma) \otimes \R^N_{\mbf{M}}}, \quad k \in \N.
\end{equation}
\end{cor}

\begin{proof}
Interpret Theorem \ref{thm:iterconv} in the context of Algorithm \ref{alg:sii}. The bound $\phi_{\min} \ge s_*$ is a consequence of Lemma \ref{lmm:sinversion}.
\end{proof}

Obviously the previous Corollary has practical value only if $\psi_{\max} < s_*$. Here we will briefly discuss the value of $\psi_{\max}$ in the case that $(\hat{s}, \hv) \in \R^P \times \R^{PN}$ is associated to the dominant fixed point of Algorithm \ref{alg:siit}. Observe that the equation $\hat{\mbf{z}} = \mbf{R}(\hat{s},\hv)\hw$ is equivalent to the system
\begin{equation}
\left\{ \begin{array}{l}
\mbf{z}(y) = \w(y) - P_{\mc{A}} (t \u_{\mc{A}})(y) \\
P_{\mc{A}} (s t)(y) = P_{\mc{A}} \left( \langle \vv, \w \rangle_{\R^N_{\mbf{M}}} \right) \!(y)
\end{array} \right.
\end{equation}
for all $y \in \Gamma$. We see that, if $\w = \vv$ then $\mbf{z} = 0$, whereas, if $\langle \w(y), \vv(y) \rangle_{\R^N_{\mbf{M}}} = 0$ for all $y \in \Gamma$ then $\mbf{z} = \w$. Thus, the matrix $\mbf{R}(\hat{s},\hv)$ acts as a deflation that shrinks vectors that are close to $\vv(y)$ and preserves vectors that are almost orthogonal to $\vv(y)$. From Proposition \ref{prp:pwestimate} we know that $s^{-1}(y)$ is an approximation of the smallest eigenvalue of $\mbf{M}^{-1}\mbf{K}(y)$ and $\vv(y)$ is an approximation of the corresponding eigenvector. By Lemma \ref{lmm:sinversion} the operator norm of $\mbf{S}^{-1}$ is bounded by $\sup_{y \in \Gamma} \lambda_1^{-1}(y)$, where $\lambda_1(y)$ is the smallest eigenvalue of $\mbf{M}^{-1}\mbf{K}(y)$. Analogously, since the eigenvector corresponding to this smallest eigenvalue is deflated by $\mbf{R}(\hat{s},\hv)$, we expect the norm of $\mbf{R}(\hat{s},\hv)\mbf{S}^{-1}$ to be bounded by a value close to $\sup_{y \in \Gamma} \lambda_2^{-1}(y)$, where $\lambda_2(y)$ is the second smallest eigenvalue of $\mbf{M}^{-1}\mbf{K}(y)$. With this reasoning, if the deflation is sufficient, there is $\psi_{\max}^* \in \R$ such that
\begin{equation}
\label{eq:deflation}
\frac{\psi_{\max}}{s_*} \le \frac{\psi^*_{\max}}{s_*} \approx \lambda_{1/2} := \frac{\sup_{y \in \Gamma} \lambda_2^{-1}(y)}{\inf_{y \in \Gamma} \lambda_1^{-1}(y)} = \frac{\sup_{y \in \Gamma} \lambda_1(y)}{\inf_{y \in \Gamma} \lambda_2(y)}.
\end{equation}
The speed of convergence of the spectral inverse iteration is now characterized by what is essentially the largest ratio of the two smallest eigenvalues of the parametric matrix $\mbf{M}^{-1}\mbf{K}(y)$.

\subsubsection{Combined error analysis}

Let $(\mu, u) \in L^2_{\nu}(\Gamma) \times L^2_{\nu}(\Gamma) \otimes H_0^1(D)$ be the smallest eigenvalue and the associated eigenfunction of the continuous problem \eqref{eq:parametricevp}. Let $(\mu_h, u_h) \in L^2_{\nu}(\Gamma) \times L^2_{\nu}(\Gamma) \otimes V_h$ be the corresponding eigenpair of the semi-discrete problem \eqref{eq:dvarform}. Assume that there exists a dominant fixed point $\u_{\mc{A}} \in \W^N$ of Algorithm \ref{alg:sii} and an associated eigenvalue approximation $\mu_{h, \mc{A}} := \mu_{\mc{A}} \in \W$ as in Proposition \ref{prp:pwestimate}. Denote by $\u^{(k)} \in \W^N$ the $k$:th iterate of Algorithm \ref{alg:sii} and by $\mu_{h, \mc{A}, k} := \mu^{(k)} \in \W$ the associated solution to \eqref{eq:eigenvalue}. Let $u_{h, \mc{A}}$ and $u_{h, \mc{A}, k}$ denote the functions in $\W \otimes V_h$, whose coordinates are defined by the vectors $\u_{\mc{A}}$ and $\u^{(k)}$ respectively. The spatial, stochastic, and iteration errors may now be separated in the following sense:
\begin{equation}
\label{eq:evalerrors}
|| \mu - \mu_{h, \mc{A}, k} ||_{L^2_{\nu}(\Gamma)} \le || \mu - \mu_h ||_{L^2_{\nu}(\Gamma)} + || \mu_h - \mu_{h,\mc{A}} ||_{L^2_{\nu}(\Gamma)} + || \mu_{h,\mc{A}} - \mu_{h, \mc{A}, k} ||_{L^2_{\nu}(\Gamma)}
\end{equation}
and
\begin{align}
\label{eq:evecerrors}
|| u - u_{h, \mc{A}, k} ||_{L^2_{\nu}(\Gamma) \otimes L^2(D)} \le & || u - u_h ||_{L^2_{\nu}(\Gamma) \otimes L^2(D)} + || u_h - u_{h, \mc{A}} ||_{L^2_{\nu}(\Gamma) \otimes L^2(D)} \nonumber \\
& \quad + || u_{h, \mc{A}} - u_{h, \mc{A}, k} ||_{L^2_{\nu}(\Gamma) \otimes L^2(D)}.
\end{align}

Under sufficient conditions we may now bound each term in the equations \eqref{eq:evalerrors} and \eqref{eq:evecerrors} separately using the theory developed earlier in this section. The first term may be approximated using Theorem \ref{thm:herror}, the second term may be approximated using Theorem \ref{thm:fpconv} and Proposition \ref{prp:exponentialconv}, and the third term may be approximated using Corollary \ref{cor:iterconv} of Theorem \ref{thm:iterconv} and the hypothesis \eqref{eq:deflation}. We therefore expect that, with an optimal choise the multi-index sets $\mc{A}_{\epsilon}$ for $\epsilon > 0$, the output of the spectral inverse iteration converges to the exact solution according to
\begin{equation}
\label{eq:evecconv}
|| u - u_{h, \mc{A}_{\epsilon}, k} ||_{L^2_{\nu}(\Gamma) \otimes L^2(D)} \lesssim h^{1+l} + (\# \mc{A}_{\epsilon} )^{-r} + \lambda_{1/2}^k
\end{equation}
and similarly
\begin{equation}
\label{eq:evalconv}
|| \mu - \mu_{h, \mc{A}_{\epsilon}, k} ||_{L^2_{\nu}(\Gamma)} \lesssim h^{2l} +( \# \mc{A}_{\epsilon} )^{-r} + \lambda_{1/2}^k
\end{equation}
for certain rates $r > 0$ and $l > 0$.

\subsection{Numerical examples}
\label{ssec:siinumex}

We present numerical evidence to verify the equations \eqref{eq:evecconv} and \eqref{eq:evalconv}. In each of the following examples we compute the smallest eigenvalue and the corresponding eigenfunction of the model problem \eqref{eq:modelproblem} in the unit square $D = [0,1]^2$ using the Algorithm \ref{alg:sii}. We use the smallest eigenvector at $y = 0$ as an initial guess. For the diffusion coefficient we assume the form \eqref{eq:klexpansion} with $a_0 := 1$ and
\[
a_m(x) := \left\{ \begin{array}{l} (m + 1)^{- \varsigma} \sin(m \pi x_1), \quad m = 1,3,\ldots \\ (m + 1)^{- \varsigma}\sin(m \pi x_2), \quad m = 2,4,\ldots \end{array} \right. \quad x = (x_1,x_2) \in D,
\]
where we set $\varsigma = 3.2$. Now $|| a_m ||_{L^{\infty}(D)} \le C m^{-\varsigma}$ and $|| a_m ||_{W^{2,\infty}(D)} \le C m^{-\varsigma + 2}$ so that the assumptions \eqref{eq:positiveness} -- \eqref{eq:psummability2} for $s=2$ are satisfied with $p_0 > \varsigma^{-1}$ and $p_2 > (\varsigma - 2)^{-1}$. We therefore expect the regions of analyticity in Proposition \ref{prp:analyticity} to increase according to $\tau_m \ge C m^{\varsigma - 1}$.

The deterministic mesh is a uniform grid of second order quadrilateral elements in all computations. The discretization in the parameter space is obtained by defining the multi-index sets
\[
\mc{A}_{\epsilon} := \left\{ \a \in (\N^{\infty}_0)_c \ \left\vert \ \prod_{m \in \supp{\a}} \eta_m^{\a_m} > \epsilon \right. \right\},
\]
where we set
\[
\eta_m := \left(\tau_m + \sqrt{1 + \tau_m^2} \right)^{-1}, \quad \tau_m := (m+1)^{\varsigma-1}, \quad m=1,2,\ldots
\]
as in the proof of Proposition \ref{prp:exponentialconv}. Multi-index sets of this form have been introduced in \cite{bierischwab09} and in \cite{bieriandreevschwab09} an algorithm for generating them has been suggested.

We use a matrix free formulation of the conjugate gradient method for solving the linear systems \eqref{eq:linearsys} and \eqref{eq:blocklinearsys}. The preconditioner is constructed using the mean of the parametric matrix in question \cite{Powell:2008hc} and as an initial guess we set the solution of the system from the previous iteration. We wish to note that in this setting only a very few iterations of the conjugate gradient method are needed at each step of the spectral inverse iteration.

In the lack of an exact solution we compute an overkill solution $(\mu_*,u_*)$ for which the number of deterministic degrees of freedom is $N = 36 741$, the parameter $\epsilon$ is chosen such that $\# \mc{A}_{\epsilon} = 264$, and the number of iterations is $k = 16$. This results in roughly $10^7$ total degrees of freedom. The number of active dimensions in the overkill solution is $M(\mc{A}_{\epsilon}) = 113$. All the numerical examples in this section have been computed using this overkill solution as a reference. The expected value and variance of the eigenfunction are presented in Figure \ref{fig:normandvar}.

\begin{figure}[htb]
\begin{center}
\subfloat[{Mean $\E[u_*(x,\cdot)]$}]{\includegraphics[width=0.48\textwidth]{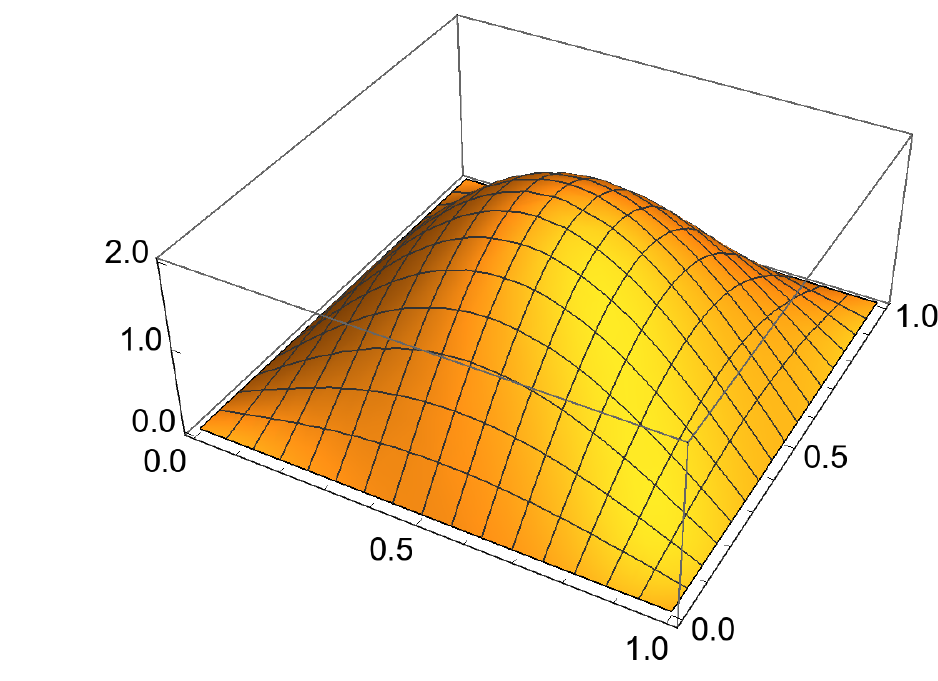}}\quad
\subfloat[{Variance $\Var[u_*(x,\cdot)]$}]{\includegraphics[width=0.48\textwidth]{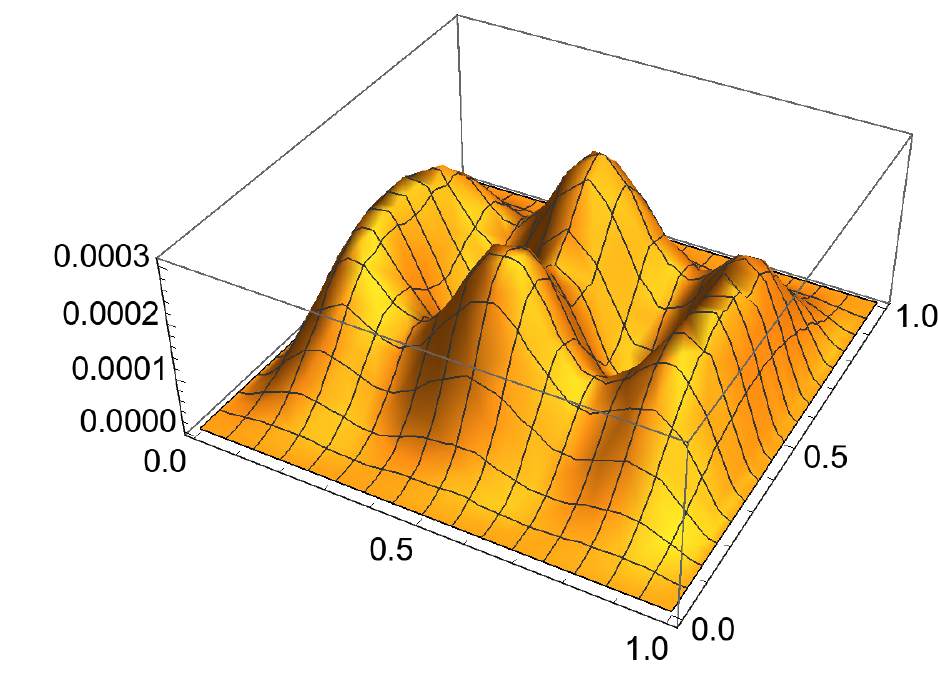}}
\caption{The mean and variance of the eigenfunction as computed by Algorithm \ref{alg:sii}.}
\label{fig:normandvar}
\end{center}
\end{figure}

\subsubsection{Convergence in space}

Keeping the number of stochastic degrees of freedom $\# \mc{A}_{\epsilon} = 264$ and the number of iterations $k = 16$ fixed, we may investigate the convergence of the solution $(\mu_{*,h},u_{*,h})$ as a function of the spatial discretization parameter $h$. This convergence for piecewise quadratic basis functions is illustrated in Figure \ref{fig:spatialconv}. We observe algebraic convergence rates of order 3 and 4 for the eigenfunction and eigenvalue respectively, exactly as predicted by Theorem \ref{thm:herror}. Thus, the error behaves like $N^{-3/2}$ and $N^{-2}$ with respect to the number of deterministic degrees of freedom.

\begin{figure}[htb]
\begin{center}
\subfloat[{The errors $|| u_* - u_{*,h}||_{L^2_{\nu}(\Gamma) \otimes L^2(D)}$.}]{\includegraphics[width=0.48\textwidth]{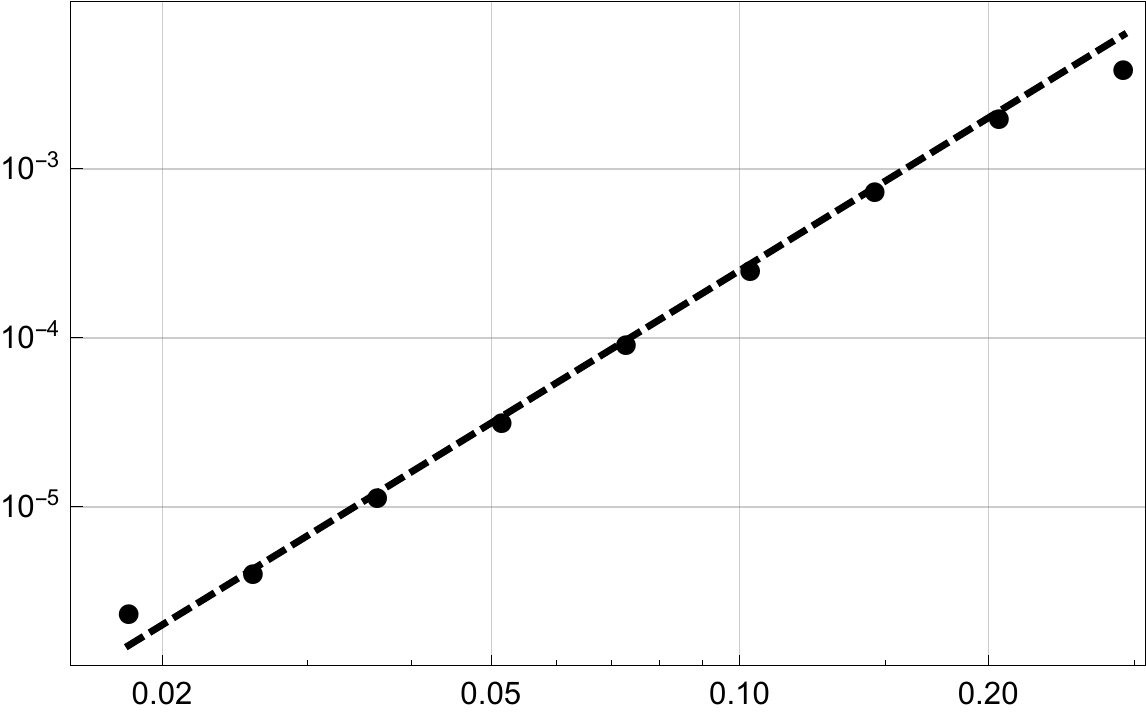}}\quad
\subfloat[{The errors $|| \mu_* - \mu_{*,h}||_{L^2_{\nu}(\Gamma)}$.}]{\includegraphics[width=0.48\textwidth]{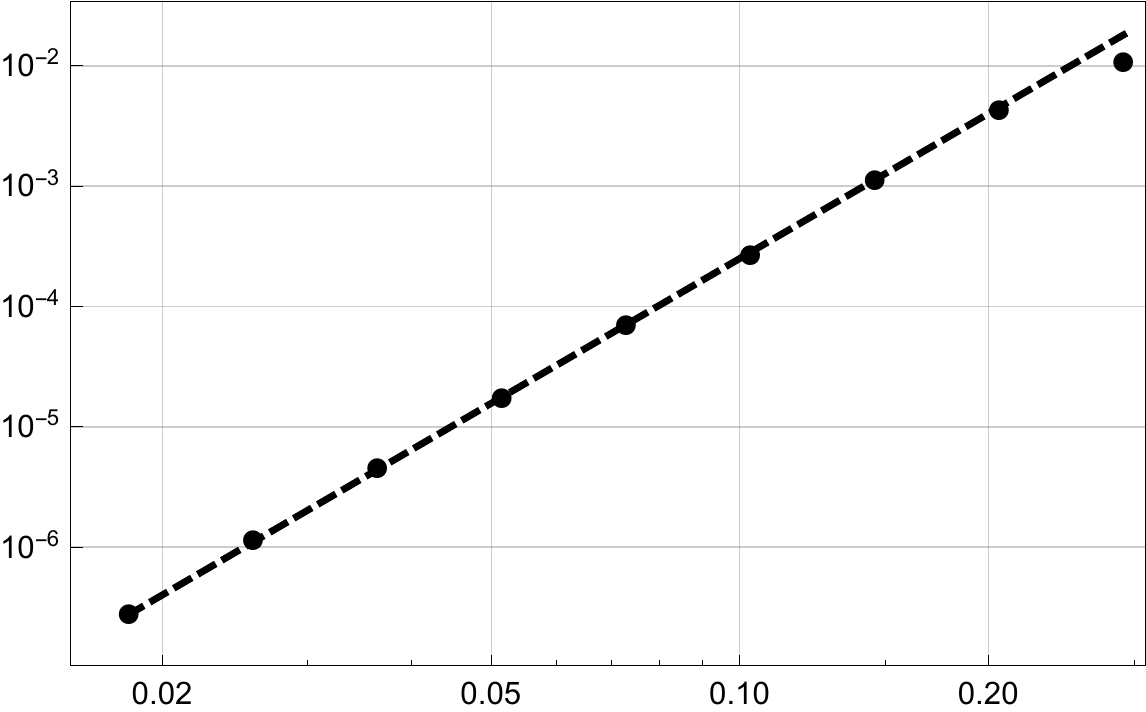}}
\caption{Convergence of the spatial errors for the eigenfunction and eigenvalue as computed by Algorithm \ref{alg:sii}. The points represent a log-log plot of the errors as a function of $h$. The dashes lines represent the rates $h^3$ and $h^4$ respectively.}
\label{fig:spatialconv}
\end{center}
\end{figure}

\subsubsection{Convergence in the parameter domain}

Keeping the number of spatial degrees of freedom $N = 36 741$ and the number of iterations $k = 16$ fixed, we may investigate the convergence of the solution $(\mu_{*,\mc{A}_{\epsilon}},u_{*,\mc{A}_{\epsilon}})$ as a function of $\# \mc{A}_{\epsilon}$ as $\epsilon \to 0$. This convergence is illustrated in Figure \ref{fig:stochasticconv}. We observe approximate algebraic convergence rates of order $-r = -1.9$ with respect to the number of stochastic degrees of freedom $\# \mc{A}_{\epsilon}$.

\begin{figure}[htb]
\begin{center}
\subfloat[{The errors $|| u_* - u_{*,\mc{A}_{\epsilon}}||_{L^2_{\nu}(\Gamma) \otimes L^2(D)}$.}]{\includegraphics[width=0.48\textwidth]{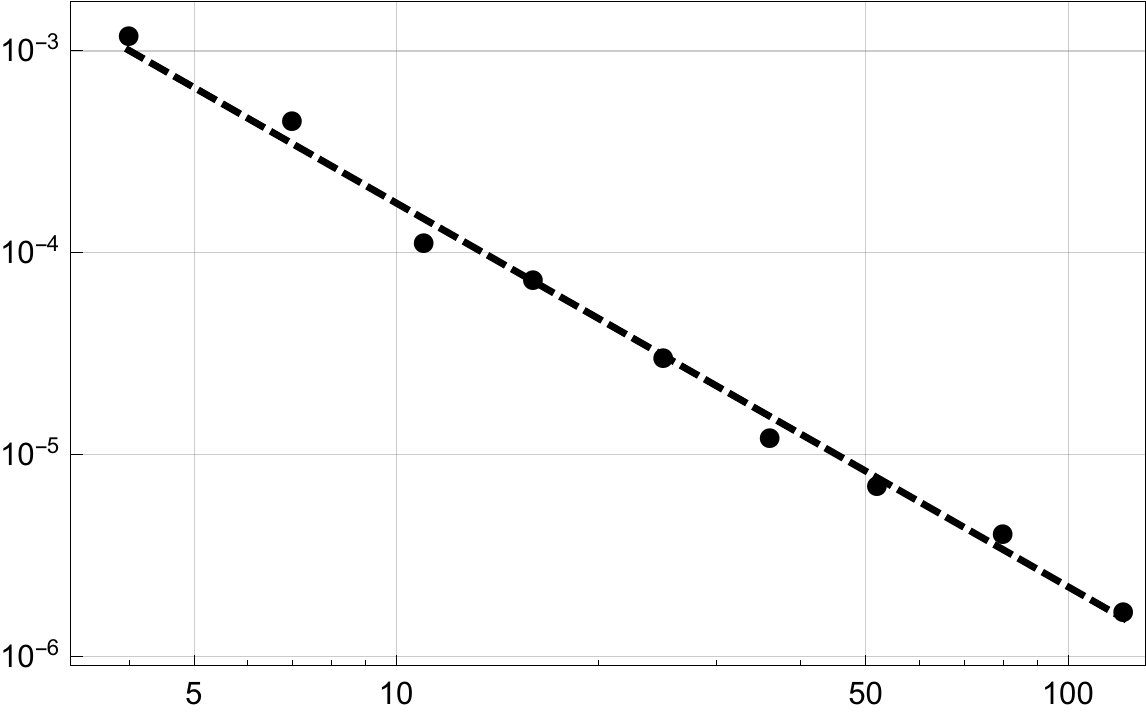}}\quad
\subfloat[{The errors $|| \mu_* - \mu_{*,\mc{A}_{\epsilon}}||_{L^2_{\nu}(\Gamma)}$.}]{\includegraphics[width=0.48\textwidth]{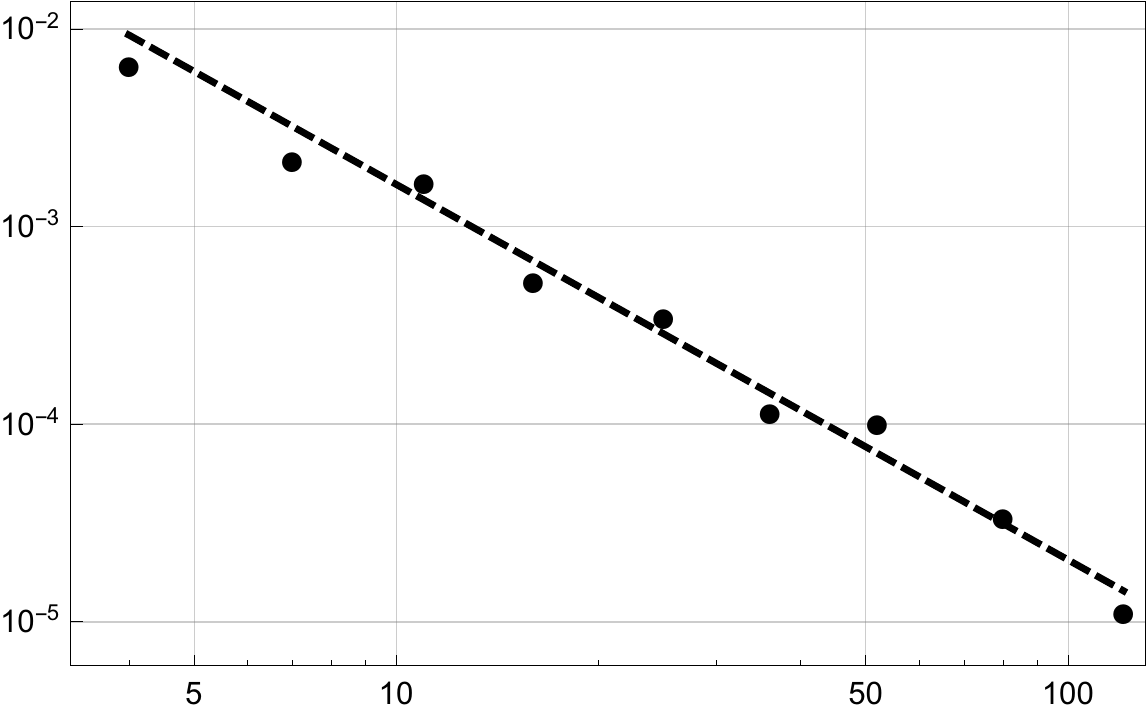}}
\caption{Convergence of the stochastic errors for the eigenfunction and eigenvalue as computed by Algorithm \ref{alg:sii}. The points represent a log-log plot of the errors as a function of $\# \mc{A}_{\epsilon}$. The dashed lines represent the rate $(\# \mc{A}_{\epsilon})^{-1.9}$.}
\label{fig:stochasticconv}
\end{center}
\end{figure}

In Figure \ref{fig:legendresunsorted} we have presented the norms of the Legendre coefficients of the overkill solution. The ordering of the coefficients is the same as the order in which they would appear in the multi-index set $\# \mc{A}_{\epsilon}$ as $\epsilon \to 0$. We see that the norms converge at the rate $-r - 1/2 = -2.4$ exactly as we would expect from the proof of Proposition \ref{prp:exponentialconv}. In Figure \ref{fig:legendressorted} we have presented the norms of the same Legendre coefficients sorted by decreasing magnitude. From this Figure we estimate that, with an optimal selection of the multi-index sets we could in fact observe a rate of convergence $-r = -2.3$ for the error of the solution. This ideal rate of convergence is somewhat faster than the asymptotic theoretical bound of $-r = -\varsigma + 3/2 = -1.7$ predicted by Proposition \ref{prp:exponentialconv}.

\begin{figure}[htb]
\begin{center}
\subfloat[{The values $|| (u_*)_{\a} ||_{L^2(D)}$.}]{\includegraphics[width=0.48\textwidth]{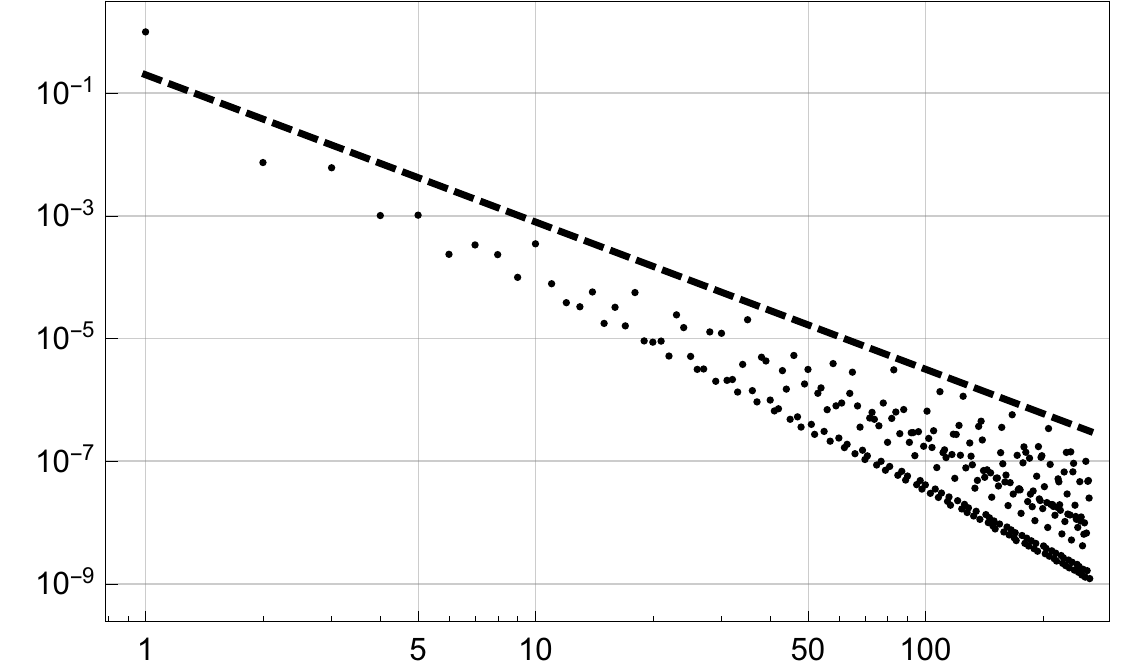}}\quad
\subfloat[{The values $| (\mu_*)_{\a} |$.}]{\includegraphics[width=0.48\textwidth]{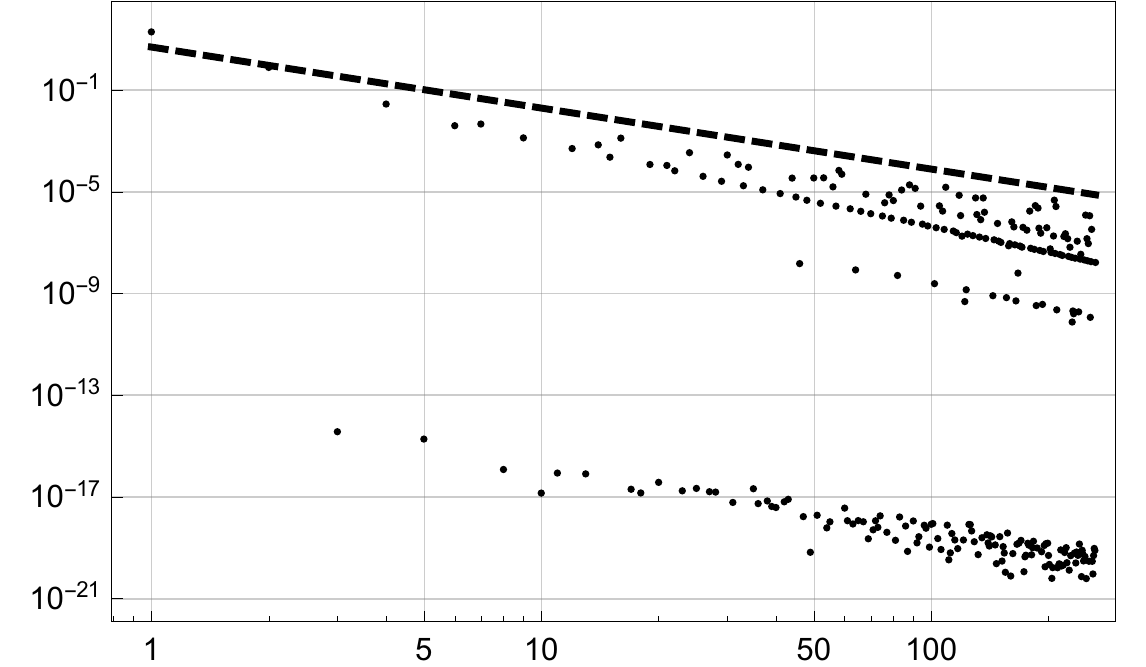}}
\caption{A log-log plot of the norms of the Legendre coefficients of the overkill solution. The dashed lines represent the algebraic rate $-2.4$.}
\label{fig:legendresunsorted}
\end{center}
\end{figure}

\begin{figure}[htb]
\begin{center}
\subfloat[{The values $|| (u_*)_{\a} ||_{L^2(D)}$.}]{\includegraphics[width=0.48\textwidth]{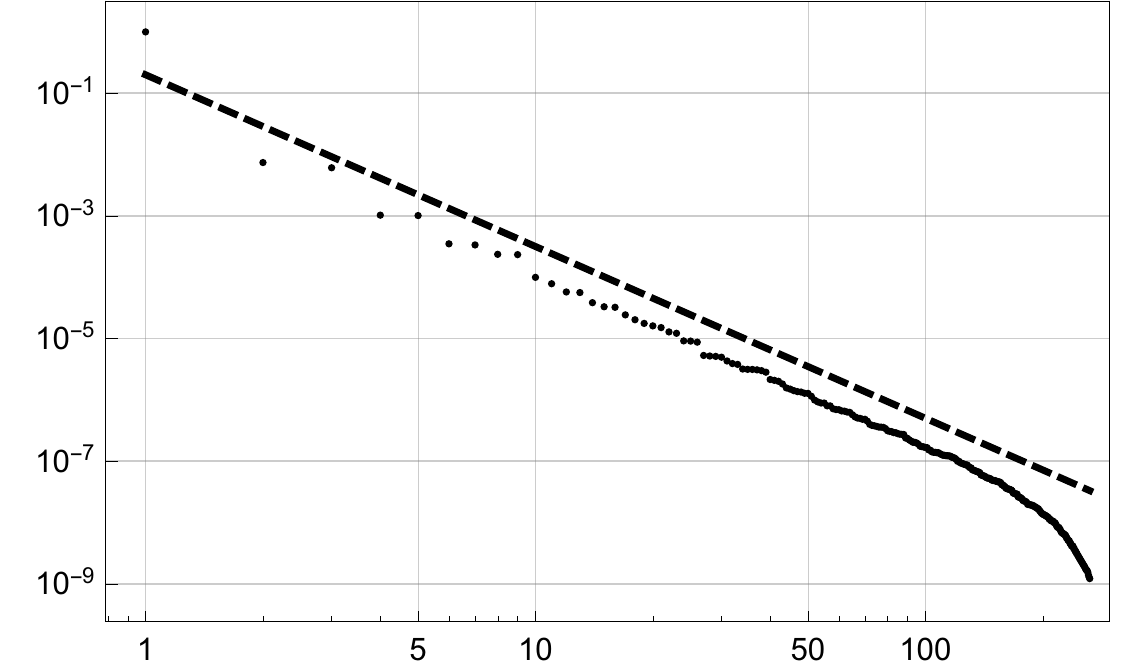}}\quad
\subfloat[{The values $| (\mu_*)_{\a} |$.}]{\includegraphics[width=0.48\textwidth]{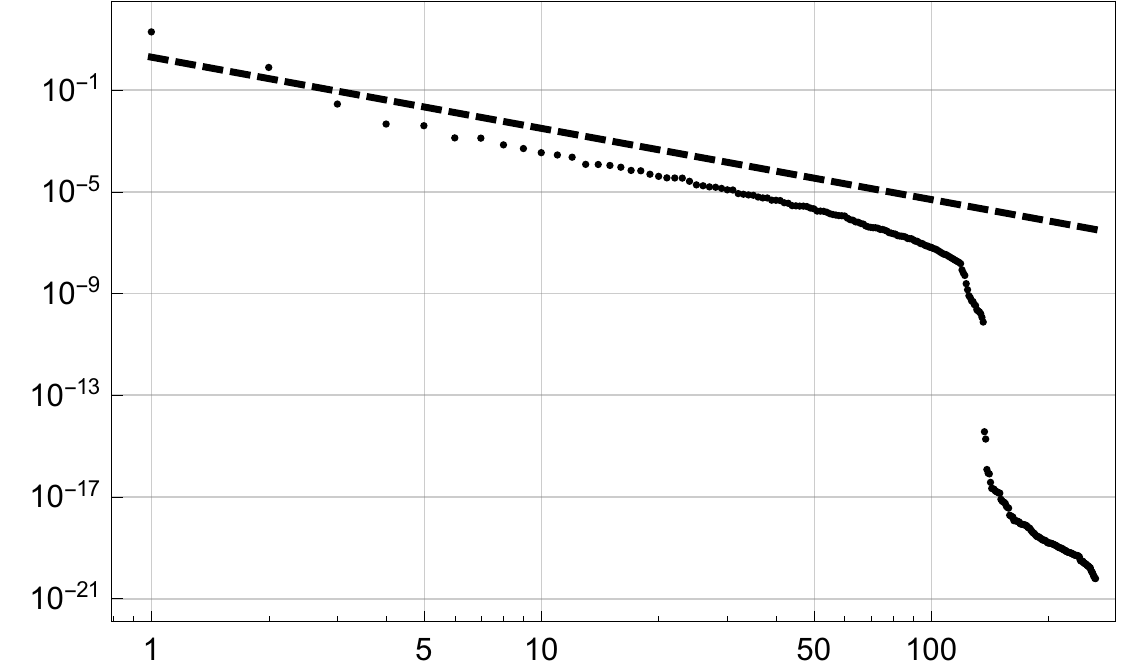}}
\caption{A log-log plot of the norms of the Legendre coefficients of the overkill solution sorted by decrasing magnitude. The dashed lines represent the algebraic rate $-2.8$.}
\label{fig:legendressorted}
\end{center}
\end{figure}

Interestingly we observe two well separated clusters of values in Figure \ref{fig:legendresunsorted}b. It seems that many of the multi-indices that correspond to relatively large Legendre coefficients of the eigenfunction, account only for a marginal contribution to the eigenvalue.

\subsubsection{Convergence of the iteration error}

Keeping the number of spatial basis functions $N = 36 741$ and the parameter $\epsilon$ fixed so that $\# \mc{A}_{\epsilon} = 264$, we may investigate the convergence of the solution $(\mu_{*,k},u_{*,k})$ as a function of the number of iterations $k$. This convergence is illustrated in Figure \ref{fig:iterationconv}. Assuming that the variation in the eigenvalues within the parameter space is small, the value $\lambda_{1/2}$ defined in \eqref{eq:deflation} may be approximated by the ratio of the two smallest eigenvalues of the problem at $y = 0$. Thus, Figure~\ref{fig:iterationconv} suggests that the error behaves asymptotically like $\lambda_{1/2}^k$, just as predicted by Corollary \ref{cor:iterconv}.

It is worth noting that, from the analysis of the classical inverse iteration, one might expect the eigenvalue to converge faster than the eigenfunction. In fact, the eigenvalue exhibits a faster rate of convergence at first and the error behaves like $\lambda_{1/2}^{2k}$. Comparing to the results of the previous example, we see that $k \approx 9$ represents a turning point after which the stochastic error in the eigenfunction starts to dominate the iteration error. Hence, for $k \ge 9$ the polynomial approximation in the parameter domain is insufficient to guarantee the degree of accuracy that is required for the eigenvalue to exhibit the faster rate of convergence that is otherwise characteristic to it.

\begin{figure}[htb]
\begin{center}
\subfloat[{The errors $|| u_* - u_{*,k}||_{L^2_{\nu}(\Gamma) \otimes L^2(D)}$.}]{\includegraphics[width=0.48\textwidth]{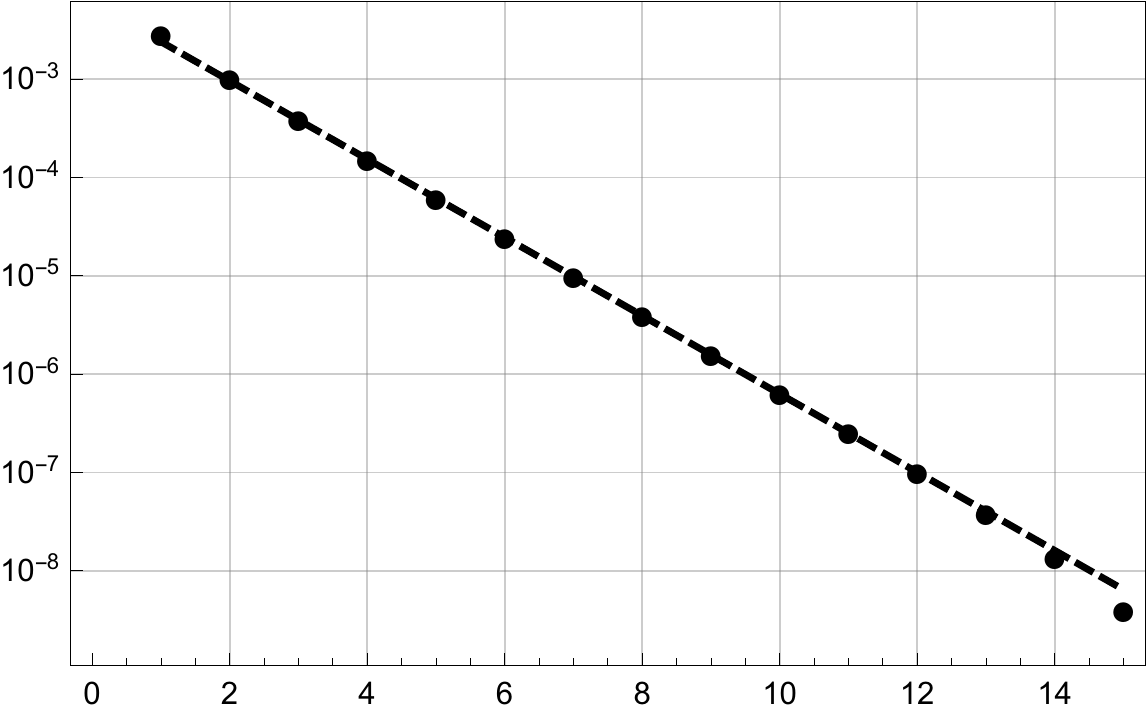}}\quad
\subfloat[{The errors $|| \mu_* - \mu_{*,k}||_{L^2_{\nu}(\Gamma)}$.}]{\includegraphics[width=0.48\textwidth]{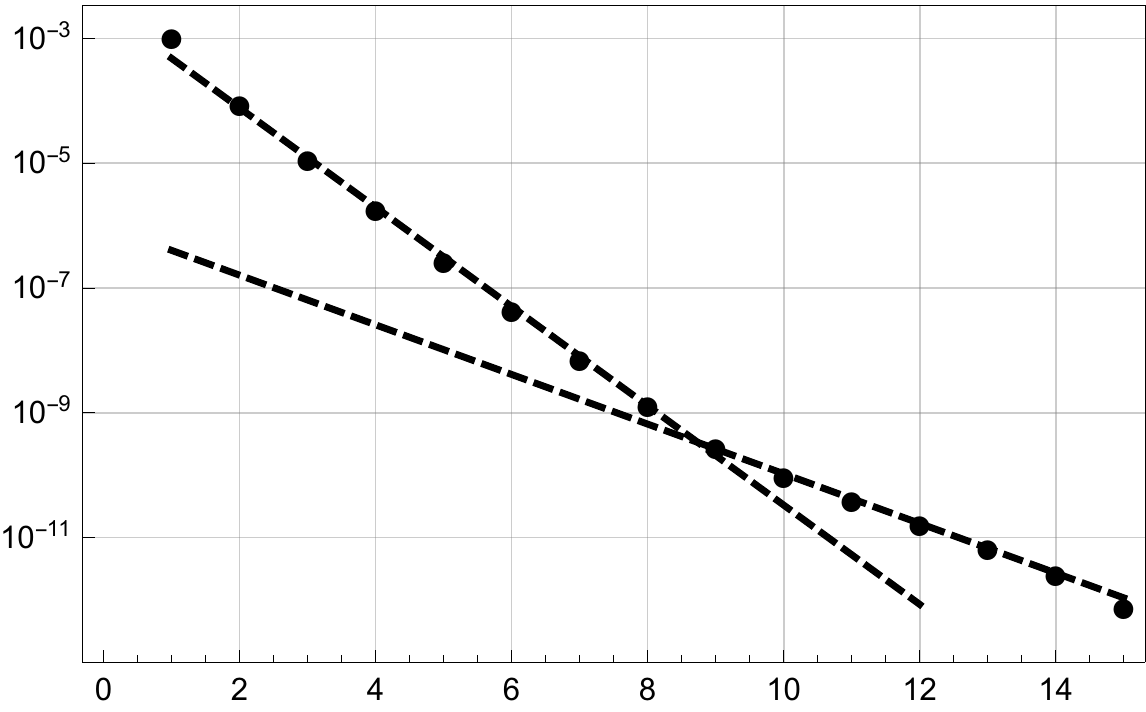}}
\caption{Convergence of the iteration errors for the eigenfunction and eigenvalue as computed by Algorithm \ref{alg:sii}. The points represent a log plot of the errors as a function of $k$. The dashed lines represent the rates $\bar{\lambda}_{1/2}^k$ and $\bar{\lambda}_{1/2}^{2k}$, where $\bar{\lambda}_{1/2}$ is the ratio of the two smallest eigenvalues of the problem at $y = 0$.}
\label{fig:iterationconv}
\end{center}
\end{figure}

\subsubsection{Concluding remarks and comparison to sparse collocation}

Using the finest levels of discretization, i.e., $N = 9296$ degrees of freedom for approximation in space and $\# \mc{A}_{\epsilon} = 121$ degrees of freedom for approximation in the parameter domain, and computing $k = 9$ steps of the inverse iteration we obtain a solution for which the $L^2_{\nu}(\Gamma) \otimes L^2(D)$ error of the eigenfunction is approximately $3 \cdot 10^{-6}$. The number of total degrees of freedom in this case is more than $10^6$ and the number of active dimensions is $M(\mc{A}_{\epsilon}) = 60$. The total computational time on a standard desktop machine is approximately five minutes, most of which is spent in the conjugate gradient method for the linear systems \eqref{eq:linearsys} and \eqref{eq:blocklinearsys}.

When the solution computed via the spectral inverse iteration is compared to the results of the non-composite version of the sparse collocation method introduced in \cite{bieri11} and employed in e.g. \cite{andreevschwab12} (see equations (5.12) -- (5.13) and (5.16) -- (5.17)), the statistics of the two solutions seem to almost coincide. Again using the finest levels of discretization ($N = 9296$ and $\# \mc{A}_{\epsilon} = 121$) for both methods, the $L^2(D)$ errors of mean and variance of the eigenfunction are both less than $3 \cdot 10^{-8}$ and the errors in the eigenvalue are less than $3 \cdot 10^{-11}$ and $3 \cdot 10^{-9}$ for the mean and variance respectively.

\section{Spectral subspace iteration}
\label{sec:subspaceiteration}

In this section we extend the spectral inverse iteration to a spectral subspace iteration, with which we can compute dominant subspaces of the inverse of the parametric matrix under consideration. The underlying assumption is that the subspace is sufficiently smooth with respect to the parameters. Convergence of the spectral subspace iteration is verified through numerical experiments.

\subsection{Algorithm description}

As with the classical subspace iteration, the idea in the spectral version is to perform inverse iteration for a set of vectors and orthogonalize these vectors at each step. Orthogonality should here be understood in a sense that the vectors are orthogonal for all points in the parameter space $\Gamma$. This can be approximately achieved by performing the Gram-Schmidt orthogonalization process for the vectors in the Galerkin sense, i.e., by projecting each elementary operation to the basis $\W$.

Fix a finite set of multi-indices $\mc{A} \subset (\mbb{N}_0^{\infty})_c$ and let $P= \# \mc{A}$. The spectral subspace iteration for the system \eqref{eq:matrixevp} is now defined in Algorithm \ref{alg:ssi}. Observe that, if the projections were precise, then the Algorithm would correspond to performing the classical subspace iteration pointwise on $\Gamma$. Orhtogonalization of the basis vectors via the Gram-Schmidt process is achieved in step (2). We expect Algorithm \ref{alg:ssi} to converge to an approximate basis for the $Q$-dimensional invariant subspace associated to the smallest eigenvalues of the system.

\begin{alg}[Spectral subspace iteration]
\label{alg:ssi}
Fix $tol > 0$ and let $\{ \u^{(0,q)} \}_{q=1}^Q \subset \W^N$ be an initial guess for the basis of the subspace. For $k = 1,2,\ldots$ do
\begin{enumerate}
\item[(1)] For each $q = 1, \ldots, Q$ solve $\vv^{(q)} \in \W^N$ from the linear equation
\begin{equation}
P_{\mc{A}} \left( \mbf{K} \vv^{(q)} \right) = \mbf{M} \u^{(k-1,q)}.
\end{equation}
\item[(2)] For $q = 1, \ldots, Q$ do
\begin{enumerate}
\item[(2.1)] Set
\begin{equation}
\w^{(q)} = \vv^{(q)} - \sum_{i=1}^{q-1} P_{\mc{A}}\left(\u^{(k,i)} P_{\mc{A}} \left( \langle \vv^{(q)}, \u^{(k,i)} \rangle_{\R^N_{\mbf{M}}} \right) \right).
\end{equation}
\item[(2.2)] Solve $s^{(q)} \in \W$ from the nonlinear equation
\begin{equation}
P_{\mc{A}} \left((s^{(q)})^2\right) = P_{\mc{A}} \left( || \w^{(q)} ||_{\R^N_{\mbf{M}}}^2 \right).
\end{equation}
\item[(2.3)] Solve $\u^{(k,q)} \in \W^N$ from the linear equation
\begin{equation}
P_{\mc{A}} \left( s^{(q)} \u^{(k,q)} \right) = \w^{(q)}.
\end{equation}
\end{enumerate}
\item[(3)] Stop if a suitable criterion is satisfied and return $\{ \u^{(k,q)} \}_{q=1}^Q \subset \W^N$ as the approximate basis for the subspace.
\end{enumerate}
\end{alg}

In general we can not expect the output vectors $\{ \u^{(k,q)} \}_{q=1}^Q \subset \W^N$ of Algorithm \ref{alg:ssi} to converge to any particular eigenvectors of the system \eqref{eq:matrixevp}. However, we still expect them to approximately span the subspace associated to the smallest eigenvalues of the system. In view of Section \ref{sec:analyticity}, if a cluster of eigenvalues is sufficiently well separated from the rest of the spectrum, then we assume the associated subspace to be analytic with respect to the parameter vector $y \in \Gamma$. In this case we may expect optimal convergence of the projections in the Algorithm.

\begin{rmk}
In order to measure convergence of the Algorithm \ref{alg:ssi} we should be able to estimate the angle between subspaces over the parameter space $\Gamma$. It is not entirely trivial to perform this kind of a computation in practise. The numerical examples in Section \ref{ssec:ssinumex} will hopefully give some more insight on this.
\end{rmk}

\begin{rmk}
As noted in Section \ref{sec:analyticity}, the smallest eigenvalue of the problem \eqref{eq:parametricevp} is always simple, hence analytic. For more general problems this might not be the case. For instance, in the event of an eigenvalue crossing, the eigenmode corresponding to the pointwise smallest eigenvalue is not (in general) even a continuous function of the parameter vector $y$. In this case we can modify the Algorithm \ref{alg:ssi} by adding the step
\begin{itemize}[leftmargin=1cm]
\item[(2.0)] Set $\vv^{(1)} = \sum_{q=1}^Q \vv^{(q)}$
\end{itemize}
before step (2.1). This should ensure optimal convergence, since even if the eigenmodes change places, we still expect their sum to be smooth with respect to $y$.
\end{rmk}

Using the tensors defined in Section \ref{sec:inverseiteration} we may write Algorithm \ref{alg:ssi} in the following form.

\begin{alg}[Spectral subspace iteration in tensor form]
Fix $tol > 0$ and let $\{ \hu^{(0,q)} \}_{q=1}^Q \subset \R^{PN}$ be an initial guess for the basis of the subspace. For $k = 1,2,\ldots$ do
\begin{enumerate}
\item[(1)] For each $q = 1, \ldots, Q$ solve $\hv^{(q)} \in \R^{PN}$ from the linear system
\begin{equation}
\widehat{\mbf{K}} \hv^{(q)} = \widehat{\mbf{M}} \hu^{(k-1,q)}.
\end{equation}
\item[(2)] For $q = 1, \ldots, Q$ do
\begin{enumerate}
\item[(2.1)] Set
\begin{equation}
\hw^{(q)} = \hv^{(q)} - \sum_{i=1}^{q-1} \mbf{T}\left( F^v(\hv^{(q)},\hu^{(k,i)}) \right)\hu^{(k,i)}.
\end{equation}
\item[(2.2)] Solve $\hat{s}^{(q)} \in \R^P$ from the nonlinear system
\begin{equation}
F(\hat{s}^{(q)},\hw^{(q)}) = 0
\end{equation}
with the initial guess $s^{(q)}_{\a} = || \hw^{(q)} ||_{\R^P \otimes \R^N_{\mbf{M}}} \delta_{\a 0}$ for $\a \in \mc{A}$.
\item[(2.3)] Solve $\hu^{(k,q)} \in \R^{PN}$ from the linear system
\begin{equation}
\mbf{T}(\hat{s}^{(q)}) \hu^{(k,q)} = \hw^{(q)}.
\end{equation}
\end{enumerate}
\item[(3)] Stop if a suitable criterion is satisfied and return $\{ \hu^{(k,q)} \}_{q=1}^Q \subset \R^{PN}$ as the approximate basis for the subspace.
\end{enumerate}
\end{alg}

\subsection{Numerical examples}
\label{ssec:ssinumex}

We use Algorithm \ref{alg:ssi} to compute the 3-dimensional subspace associated with the smallest eigenvalues of the model problem considered in Section \ref{ssec:siinumex}. We let the deterministic mesh be a uniform grid of second order quadrilateral elements with $N = 2465$ degrees of freedom. As an initial guess we use the smallest eigenvectors of the problem at $y = 0$. In Figure \ref{fig:evals} we have presented the four smallest eigenvalues of the problem as a function of $y_1$, when $y_2, y_3, \ldots$ are held constant. We observe an eigenvalue crossing due to which the eigenvectors corresponding to the pointwise second and third smallest eigenvalues are discontinuous as functions of $y$.

\begin{figure}[htb]
\begin{center}
\subfloat[{Eigenvalues 1-4.}]{\includegraphics[width=0.48\textwidth]{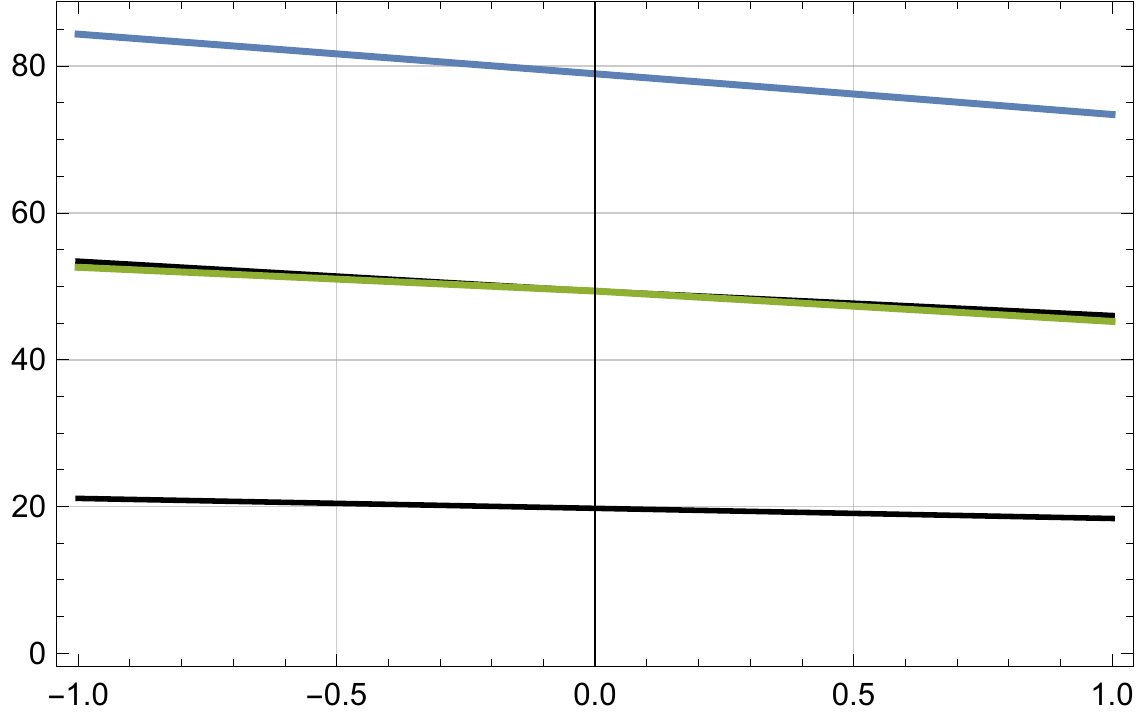}}\quad
\subfloat[{Eigenvalues 2-3.}]{\includegraphics[width=0.48\textwidth]{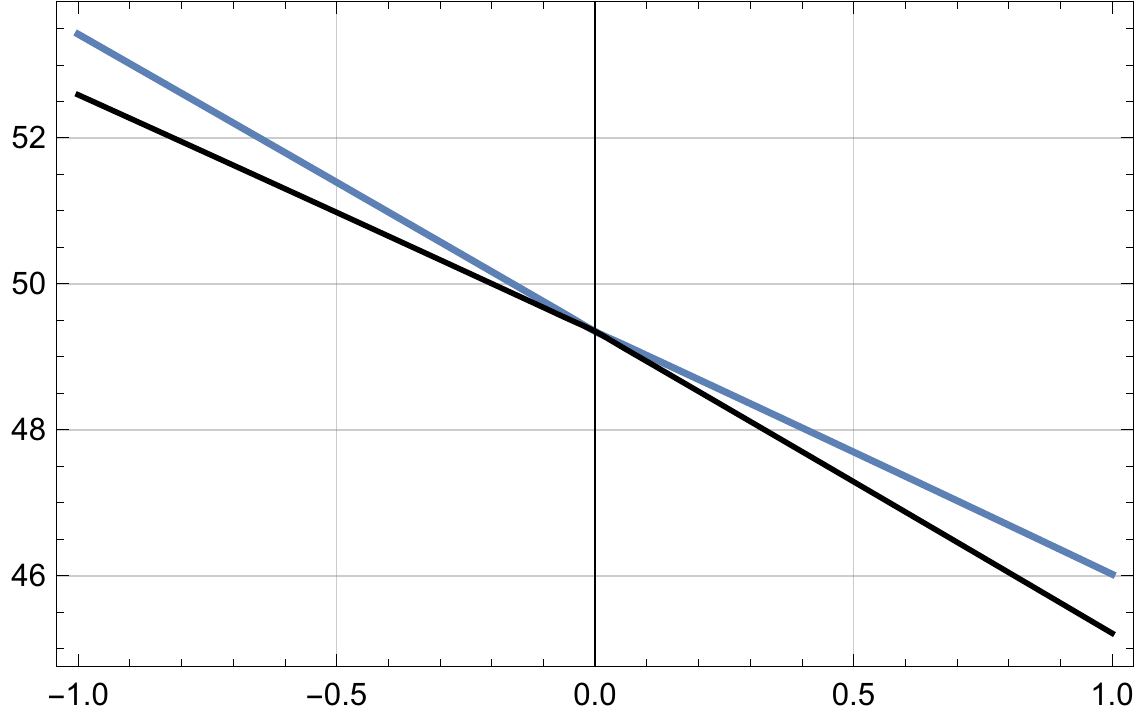}}
\caption{A few smallest eigenvalues of the model problem as a function of $y_1$ when $y_2 = y_3 = \ldots = 0$. The smallest eigenvalue is well-separated. However, we observe a crossing of the second and third smallest eigenvalues.}
\label{fig:evals}
\end{center}
\end{figure}

In order to investigate the convergence of the spectral subspace iteration, we attempt to estimate the angle between the exact invariant subspace and the approximate one computed by Algorithm \ref{alg:ssi}. For any fixed $y \in \Gamma$ we let $\vv_1(y), \ldots, \vv_Q(y)$ be a set of $\R^N_{\mbf{M}}$ orthonormal exact eigenvectors corresponding to the $Q$-smallest eigenvalues of the problem. We define
\[
\theta_k(y) := |\det(\Theta^{(k)}(y))|,
\]
where $\Theta^{(k)}(y) \in R^{Q \times Q}$ is a matrix with elements $\Theta^{(k)}_{ij}(y) = \langle \u^{(k)}_i(y), \vv_j(y) \rangle_{\R^N_{\mbf{M}}}$. Now $\theta_k(y)$ can be viewed as the cosine of the angle between the two subspaces at $y \in \Gamma$ (see for instance \cite{gunawanneswansetya-budhi05} formula (2.2)). Thus, convergence of the algorithm can be measured in terms of the statistics of $\theta_k$. In the following examples we have estimated the mean and variance of $\theta_k$ using the non-composite version of the sparse collocation operator employed in \cite{andreevschwab12}. For the definition of the collocation operator we have used the overkill multi-index set of section \ref{ssec:siinumex} ($\# \mc{A}_{\epsilon} = 264$).

Convergence of the spectral subspace iteration for $Q = 3$ is illustrated in Figure \ref{fig:subspace}. We see that the values $\arccos(\E[\theta_k])$ behave like $\lambda_{3/4}^k$, where $\lambda_{3/4}$ is the ratio of the third and fourth smallest eigenvalues of the problem. Simultaneously the values $\Var[\theta_k]$ converge to zero. These results suggest that the angle between the exact subspace and the approximation computed by Algorithm \ref{alg:ssi} converges to zero on $\Gamma$. Furthermore, the rate of convergence is characterised by the rate $\lambda_{3/4}^k$ much like for the classical subspace iteration. Note however, that with a fixed basis for polynomial approximation, i.e. a fixed multi-index set $\mc{A}_{\epsilon}$, only a certain accuracy for the output may be reached. Increasing the number of basis polynomials makes more accurate solutions achievable.

\begin{figure}[htb]
\begin{center}
\subfloat[{The values $\arccos(\E[\theta_k])$ for $\# \mc{A}_{\epsilon} = 52$.}]{\includegraphics[width=0.48\textwidth]{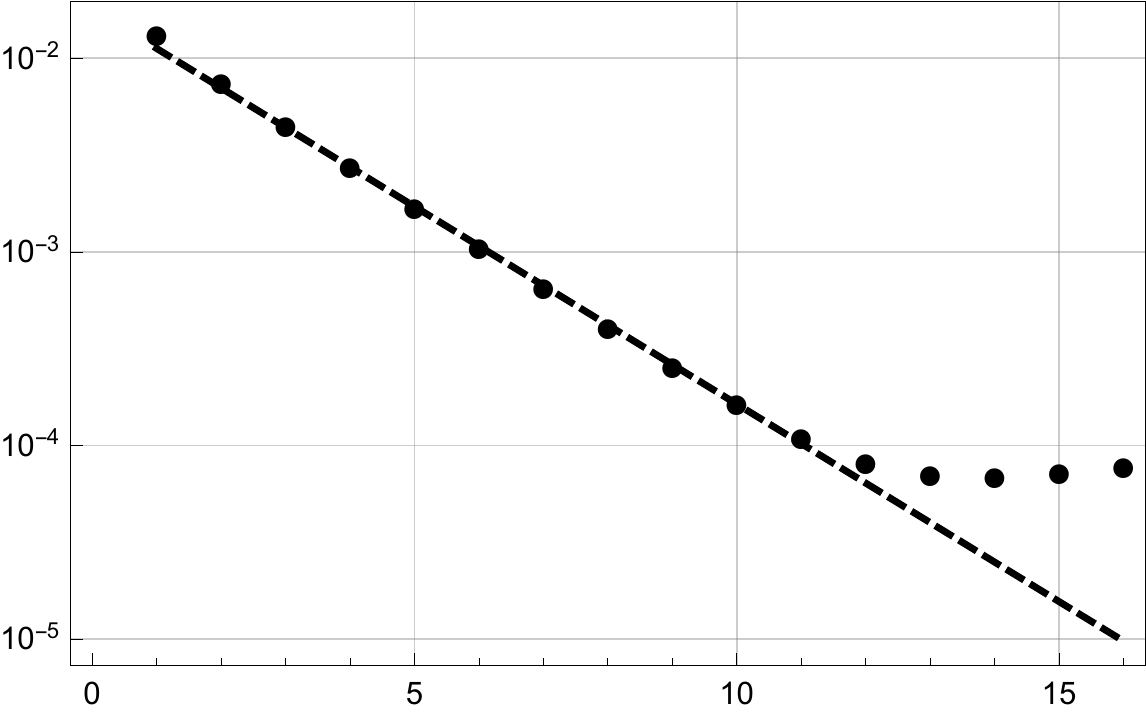}}\quad
\subfloat[{The values $\arccos(\E[\theta_k])$ for $\# \mc{A}_{\epsilon} = 121$.}]{\includegraphics[width=0.48\textwidth]{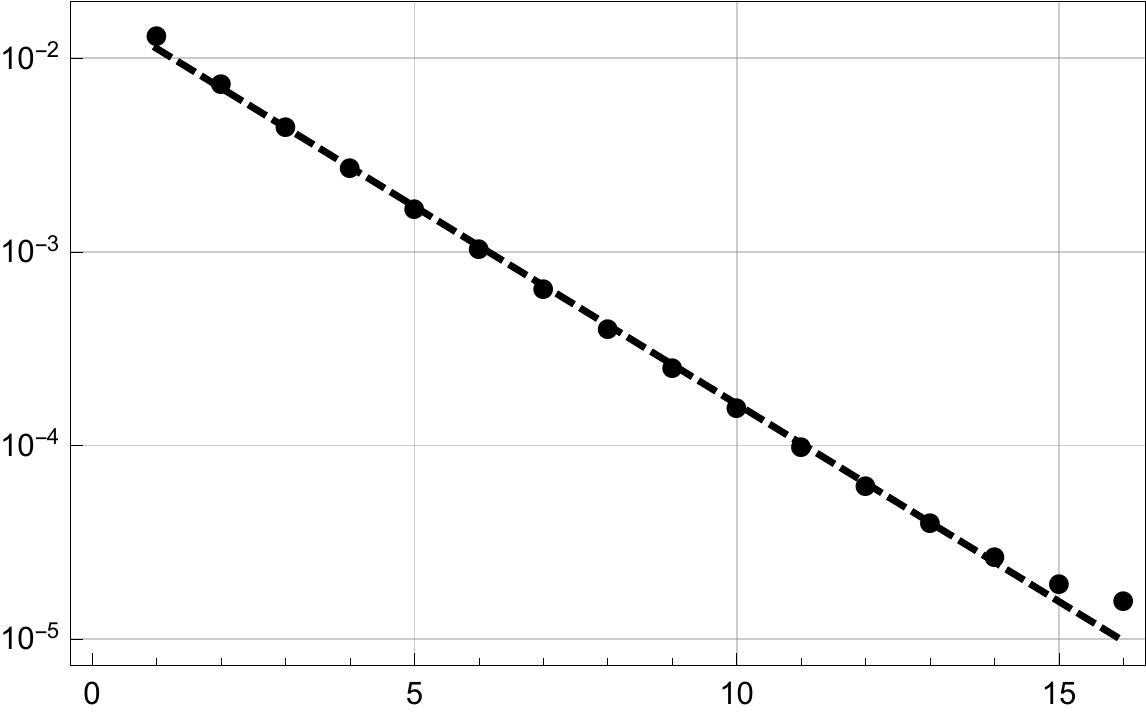}}\\
\subfloat[{The values $\Var[\theta_k]$ for $\# \mc{A}_{\epsilon} = 52$.}]{\includegraphics[width=0.48\textwidth]{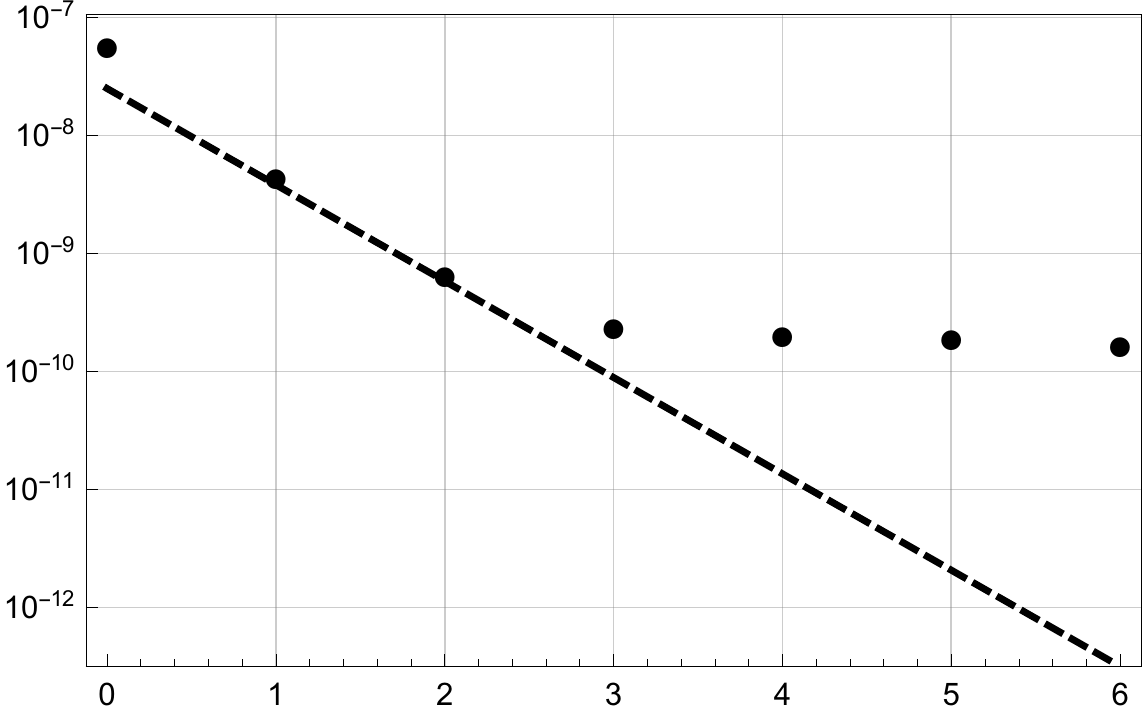}}\quad
\subfloat[{The values $\Var[\theta_k]$ for $\# \mc{A}_{\epsilon} = 121$.}]{\includegraphics[width=0.48\textwidth]{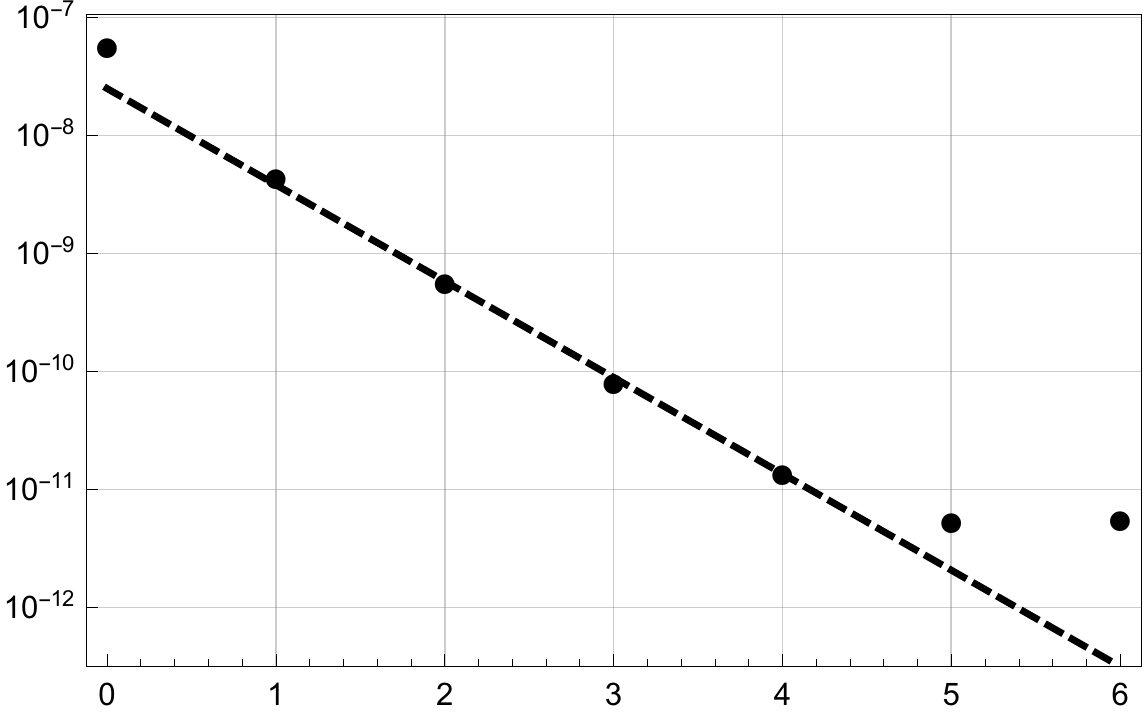}}
\caption{Convergence of the Algorithm \ref{alg:ssi} for $Q=3$. The points represent a log plot of approximate statistics of the error measure $\theta^{(k)}$ as a function of $k$. The dashed lines represent the rates $\bar{\lambda}_{3/4}^k$ and $\bar{\lambda}_{3/4}^{4k}$ for the top and bottom row plots respectively. Here $\bar{\lambda}_{3/4}$ is the ratio of the third and fourth smallest eigenvalues of the problem at $y=0$. }
\label{fig:subspace}
\end{center}
\end{figure}

\section{Conclusions and future prospects}

We have presented a comprehensive error analysis for the spectral inverse iteration, when applied to solving the ground state of a stochastic elliptic operator. We have also proposed a method of spectral subspace iteration and, using numerical examples, shown its potential in computing approximate subspaces associated to possibly clustered eigenvalues. Further analysis, both numerical and theoretical, of this algorithm is left for future research.

The numerical examples suggest that our algorithms are both accurate and efficient. However, theoretical estimates for the computational complexity are not entirely trivial to obtain as this would require information on the structure of the tensor of coeffiecients $c_{\a \b \g}$. Moreover, when iterative solvers are used, the optimal strategy is to increase the associated tolerances in the course of the iteration. We note that sparse products of the spatial and stochastic approximation spaces, as in \cite{bieriandreevschwab09}, may be applied to further reduce the computational effort, and that matrix free algorithms also allow for easy parallelization.

%
\bibliography{bibfile}

\end{document}